 \documentclass[final,3p,times]{elsarticle}

\usepackage{amssymb}
\usepackage{amsmath}
\usepackage{amsthm}
\usepackage{float}
\usepackage{subfig} 
\usepackage{algorithm}
\usepackage{xcolor}
\usepackage{cleveref}
\usepackage[T1]{fontenc}

\newtheorem{lem}{Lemma}

\journal{}

\begin{document}

\begin{frontmatter}



\title{A level-set based finite difference method for the ground state Bose-Einstein condensates in smooth bounded domains}

 \author[label1]{Hwi Lee\corref{cor1}}
 \affiliation[label1]{organization={New York Institute of Technology},
             addressline={Department of Mathematics},
             city={Old Westbury},
             postcode={11568},
             state={New York},
             country={USA}}
 \ead{hlee50@nyit.edu}

 \cortext[cor1]{Corresponding author}
 \author[label2]{Yingjie Liu}         
 \affiliation[label2]{organization={Georgia Institute of Technology},
             addressline={School of Mathematics},
             city={Atlanta},
             postcode={30332},
             state={Georgia},
             country={USA}}
\ead{yingjie@math.gatech.edu}

\begin{abstract}
We present a level-set based finite difference method to calculate the ground states of Bose Einstein condensates in domains with curved boundaries. Our method draws on the variational and level set approaches, benefiting from both of their long-standing success. More specifically, we use the normalized gradient flow, where the spatial discretization is based on the simple Cartesian grid with fictitious values in the outer vicinity of the domains. We develop a PDE-based extension technique that systematically and automatically constructs ghost point values with third-order accuracy near irregular boundaries, effectively circumventing the computational complexity of interpolation in these regions. Another novel aspect of our work is the application of the PDE-based extension technique to a nodal basis function, resulting in an explicit ghost value mapping that can be seamlessly incorporated into implicit time-stepping methods where the extended function values are treated as unknowns at the next time step. We present numerical examples to demonstrate the effectiveness of our method, including its application to domains with corners and to problems involving higher-order interaction terms.

\end{abstract}




\end{frontmatter}



\section{Introduction}
\label{int}
The Bose-Einstein condensates (BECs) exhibited by dilute bosonic gases at ultralow temperatures have been a critical research area in quantum physics since they were first experimentally realized \cite{anderson1995observation}, confirming the theoretical predictions \cite{bose1924plancks,einstein2005quantentheorie}. The importance of BECs is evidenced by their intimate connection to macroscopic quantum phenomena such as superfluidity, and significant progress has been made in experimental, theoretical, and numerical studies of the BECs \cite{ott2001bose,lieb2002proof,markowich2003}. In particular, recent years have witnessed the emergence of extensive research activities on BECs in curved geometries such as shell-shaped BECs \cite{jia2022expansion,carollo2022observation} and quantum vortices subject to geometric constraints \cite{tononi2024quantum}, and geometrically confined BECs \cite{van2020perturbation, van2025geometrically}. 

The most commonly used mathematical model for a BEC can be expressed in terms of its macroscopic wave function $u(x,t)$, which evolves according to the time-dependent, dimensionless, nonlinear Schr\"{o}dinger equation , or commonly known as the Gross-Pitaevskii equation (GP) \cite{erdHos2010derivation}
\begin{align}
i \partial_t \psi(x,t)  &= -\frac{1}{2} \Delta \psi(x,t) + V(x) \psi(x,t) + \beta |\psi(x)|^2 \psi(x), \quad x  \in \Omega , \quad t > 0 \label{eq:nls} \\
\psi(x,t) &= 0, \quad x \in \partial \Omega
\end{align}
where $\Omega \subset \mathbb{R}^d$, $d = 1,2,3$, and $\beta \in \mathbb{R}$ indicate repulsive $(\beta>0)$ and attractive $(\beta <0)$ interactions. The external trapping potential function $V: \Omega \to \mathbb{R} $ is commonly assumed to be harmonic $ V(x) = \frac{1}{2}\left( \gamma_1x^2_1 + \dots \gamma_2 x^2_d\right), \gamma_1, \dots ,\gamma_d >0 $, among many other choices used in experimental studies such as the optical lattice potential \cite{denschlag2002abose}. The conservation of mass and energy associated with equation (\ref{eq:nls}) is given by
\begin{equation}
\int_{\Omega} |\psi(x,t)|^2 dx = 1, \quad t > 0 \label{ulc}
\end{equation}
and
\begin{equation}
\mathcal{E}[\psi](t) = \mathcal{E}[\psi](0),  \quad t > 0  
\end{equation}
respectively, where the variational energy is given by  
$$
\mathcal{E}[\psi](t) = \int_\Omega \left[\frac{1}{2}\left|\nabla \psi(x,t) \right|^2 + V(x)|\psi(x,t)|^2 + \frac{\beta}{2} |\psi(x,t)|^4 dx \right].
$$

A stationary solution of equation (\ref{eq:nls}) can be obtained using the ansatz $\psi(x,t) = e^{-i\mu t} u(x)$ where $\mu \in \mathbb{R}$ is the chemical potential of the condensate and $ u:\Omega\to \mathbb{R}$. One can arrive at the nonlinear Gross-Pitaevskii eigenvalue problem (GPE)
\begin{eqnarray}
    \mu u(x) &=& \frac{1}{2}\frac{\delta \mathcal{E}[u]}{\delta u} = -\frac{1}{2} \Delta u(x) + V(x) u(x) + \beta |u(x)|^2u(x), \quad x \in \Omega
    \label{eq:gpe}
\end{eqnarray}
subject to the homogeneous Dirichlet boundary condition and the unity mass constraint (\ref{ulc}). Of particular importance is finding a ground state eigenpair $(\mu,u)$, where $u$ is given by
$$
u = \text{argmin}_{\|\phi\|_{L^2(\Omega)} = 1} \mathcal{E}[\phi].
$$ For simplicity, we assume throughout the rest of the paper that the condensate is defocusing, i.e. $\beta \geq 0$, and $V(x) \geq 0, x\in \Omega$. It can be shown \cite{cances2010numerical} that the smallest eigenvalue $\mu_g$ is associated with exactly two eigenfunctions $\pm u_g$, where $u_g(x) >0$, $x\in \Omega$ is the unique ground state of $\mathcal{E}$;  the ground state energy is given by $\mu_g - \frac{\beta}{2} \|u_g\|^4_{L^4(\Omega)}$. There are infinitely many other eigenpairs, commonly called excited states in the physics literature, corresponding to some ordering of larger values of the energy \cite{henning2020sobolev}.  

A number of numerical methods for the ground states of the BEC have been studied and continue to motivate the emergence of novel approaches such as the data-driven method \cite{bao2025computing}. One class of existing numerical approaches is based on direct minimization of the energy functional $\mathcal{E}$ by applying, for instance, Newton-type methods \cite{bao2003ground,caliari2009minimisation,wu2017regularized}. In lieu of the optimization framework, alternative numerical methods focus on solving the GPE, the Euler-Lagrange equation of $\mathcal{E}$, in terms of nonlinear eigenvalue problems \cite{cances2014perturbation,jarlebring2014inverse,adhikari2000numerical,edwards1995numerical,chang2007computing}. Another class of numerical methods is predicated on the time-dependent reformulation of the GPE into gradient flows, as opposed to the time-independent perspectives of directly minimizing the energy or directly solving the GPE. One of the representative works is the imaginary-time method \cite{chiofalo2000ground} widely used in the physics literature and mathematically justified as equivalent to the discrete normalized $L_2$-gradient flow \cite{bao2004computing}, which is given by
\begin{align}
\label{dngf}
\partial_t u(x,t)  &= -\frac{1}{2} \frac{\delta \mathcal{E}[u]}{\delta u}, \quad x  \in \Omega , \quad t_n \leq  t  < t_{n+1} \\
u(x,t_{n+1}) &= \frac{u(x,t^{-}_{n+1})}{\|u(x,t^{-}_{n+1}\|_{L^2(\Omega)}},
\end{align}
subject to the homogeneous Dirichlet boundary condition, and given some initial condition $u_0$ with $\|u_0\|_{L^2(\Omega)} = 1$. A more generalized framework of Sobolev gradient flows has been employed in several studies \cite{henning2020sobolev,garcia2001optimizing,danaila2010new}, leading to the development of novel numerical methods, such as the hybrid approach based on Riemannian optimization \cite{danaila2017computation}. We refer to \cite{henning2025gross} and references cited therein for a more comprehensive overview of numerical methods for the GPE.

In this work, we compute the ground states by applying a backward Euler-type time stepping and finite difference spatial discretization (BEFD) to the gradient flow (\ref{dngf}) as in the pioneering work \cite{bao2004computing}. A wide variety of other discretization techniques have been studied \cite{antoine2018asymptotic,yang2025energy,ming2014efficient,aftalion2004giant,gammal1999improved}, but we adopt the BEFD which still remains one of the most popular choices in practical computations due to its simplicity, efficiency, and robustness \cite{bao2012mathematical}. In order to contextualize our contributions, we stress that our BEFD method is designed for \emph{bounded} domains with curved boundaries. Most, if not all, experimental investigations are conducted in spatial domains of finite sizes, but it is customary in numerical studies to take a large enough rectangular computational domain assuming that the solutions may decay sufficiently fast away from the experimental domain boundary. Another common practice is to take relatively simple domains such as radial or polyhedral domains, should the boundary effect be taken into account as in the case of strong nonlinear self-interaction. Our proposed BEFD scheme for more complex domains is also timely in light of the growing interest in multifaceted interplay between BECs and geometries.

The main contribution of this paper is to combine the level set method \cite{osher1988fronts} and the original BEFD method to propose a new numerical method that can automatically handle curved boundaries. The level set method is a mainstay computational framework for treating complex interfaces, as evidenced by the extensive literature (see \cite{gibou2018review}, for example) on its considerable success and advances. Our work, to the best of our knowledge, is among the first lines of studies that draw on the flexibility of the level set formulation in the context of numerical simulations of BEC, further expanding a wide range of applications for the former. As a pioneering work, we focus on the simplest setting of non-rotating, one-component GPE equation (\ref{eq:gpe}) in static domains. However, given that the level set formulation was originally devised for the motion of an interface, our approach may be extended to more complicated settings of physical importance such as the BECs in dynamic domains suggested in \cite{van2025geometrically}.

Our BEFD method employs a Cartesian grid to preserve the simplicity of the original formulation, thereby circumventing the computational expense associated with mesh generation. To handle curved boundaries, we systematically adjust the finite difference approximations near the interface using a PDE-based extension technique \cite{fedkiw1999non}. This is equivalent to automatic interpolation of ghost point values located in the outer vicinity of the domain, without invoking case-by-case analysis of local boundary geometry. The novelty of our work is to extend each nodal basis function centered in the inner vicinity of the domain across the boundary, hence effectively obtain an extension \emph{operator} from the interior boundary layer data to the exterior ghost layer data. Our approach differs from existing methods, such as that in \cite{lee2023ghost}, where the PDE-based extension technique is applied to the known interior solution at the current time level to compute the corresponding ghost point values. In contrast, our focus on the underlying extension operator is driven by the semi-implicit time-stepping scheme of the original BEFD, which requires ghost point values at the next (unknown) time level. The overall efficiency of our BEFD method remains uncompromised, as the underlying extension operator—and consequently the modified finite difference formula—is computed only once and reused throughout.

The remainder of the paper is organized as follows. We describe our proposed method in detail in Section \ref{sec:nm} and present numerical results in \ref{sec:ne}. We provide concluding remarks in Section \ref{sec:conc}.


\section{Numerical Method} \label{sec:nm}
Let us consider the $2D$ setting with a uniform rectangular grid $h \mathbb{Z}^2$, where $h$ is the spatial grid size. Given a grid point $(x_i,y_j) \in \Omega$, we call it a regular point if all its four neighbors $(x_{i+1},y_{j}), (x_{i-1},y_{j}), (x_{i},y_{j+1}),(x_{i},y_{j-1})$ belong to $\Omega$; otherwise, it is called an irregular point. We define the first ghost layer as the set of grid points $(x_i,y_j) \in \Omega^c$, of which at least one of the four neighbors is an irregular point. The second ghost layer is defined as the set of grid points in $\Omega^c$ that do not belong to the first ghost layer themselves, but at least one of their four neighbors does. \textcolor{black}{We define the two layers of ghost points as they may be necessary for the finite difference approximation of the Laplacian centered at irregular and (some of) regular points}. We use $t_n = n \Delta t$ to represent the discrete time levels, and we write $u^n_{i,j}$ to represent $u(x_i,y_j,t_n)$. We delineate the various components of our numerical method by starting with the BEFD discretization of (\ref{dngf}).

\subsection{The BEFD method}
\label{sec:ndgf}
We consider the following fully discrete normalized gradient flow, 
\begin{align}
\frac{\tilde{u}_{i,j}^{n+1}-u_{i,j}^{n}}{\tilde\Delta t} &= \frac{1}{2}\mathbf{D}_h{\tilde{u}^{n+1}_{i,j}} - V_{i,j}\tilde{u}^{n+1}_{i,j} - \beta |u^n_{i,j}|^2 \tilde{u}^{n+1}_{i,j} \\
u^{n+1} &= \frac{\tilde{u}^{n+1}}{\|\tilde{u}^{n+1}\|_{2,h}}
\end{align}
Here, $\|\cdot \|_{2,h}$ denotes the numerical approximation of $\| \cdot \|_{L^2(\Omega)}$ and $\textbf{D}$ is the approximation of the Laplacian $\Delta$ to be presented shortly. Bao and Du \cite{bao2004computing} proved that the BEFD is energy decreasing and monotone for any $\Delta t > 0$ in a rectangular domain, provided that $\beta = 0$ and the standard second-order approximation of the Laplacian is used. 

In cases where $\beta \gg 0 $, we rescale the solution by $z = {\sqrt{\beta}}u$ and solve an equivalent equation
\begin{align}
\label{eq:sgp}
\frac{\tilde{z}_{i,j}^{n+1}-z_{i,j}^{n}}{\Delta t} &= \frac{1}{2}\mathbf{D}_h{\tilde{z}^{n+1}_{i,j}} - V_{i,j}\tilde{z}^{n+1}_{i,j} -  |z^n_{i,j}|^2 \tilde{z}^{n+1}_{i,j} \\
z^{n+1} &= \sqrt{\beta} \frac{\tilde{z}^{n+1}}{\|\tilde{z}^{n+1}\|_{2,h}}.
\end{align}
In \cite{bao2007energy},  a different rescaling is used to derive asymptotic approximations of the energy and chemical potential of the BEC ground states in the semi-classical regime. Their analysis is concerned with the GP equation in the whole space $\mathbb{R}^d$ subject to perturbed harmonic potentials. We also highlight alternative approaches where Fourier-based time-splitting methods have been used in lieu of rescaling \cite{bao2003numerical}. One may also circumvent rescaling the equation in the framework of finite difference methods by adapting a local refinement strategy such as the Shishkin grid \cite{shishkin1990grid} to better resolve the semi-classical asymptotic regime.

\textcolor{black}{
We provide the initial data for the gradient flow by following the approach in \cite{bao2003ground}. In the case of small $\beta$, we use the ground state of the linear Schrodinger problem which can be obtained efficiently. In the case of large $\beta$, we consider the Thomas-Fermi approximation where the Laplacian term is ignored to solve for $u$ in the resulting algebraic equation. For intermediate values of $\beta$,  we apply a continuation technique \cite{allgower1980simplicial} to successively compute the ground states starting with the linear problem. Our current work focuses on the ground states, but we note that the continuation approach may also be useful for finding excited states of the BEC for any $\beta$. }

\subsection{Signed distance function}
\label{sec:sd}
We assume the signed distance function $\phi$ is given locally near $\partial \Omega$, i.e.
$$
|\nabla \phi | = 1
$$ 
with $\phi < 0$ in $\Omega$, and $\phi \ge 0 $ otherwise. For simple geometries, analytical expressions can be easily determined and in more general cases, highly accurate approximations of $\phi$ can be computed by well-developed algorithms such as \cite{zhao2005fast,pan2018high} in the vicinity of the interface $\phi = 0$. We can then readily compute  the unit normal $\mathbf{n} $ and the curvature $\mathbf{\kappa}$ 
$$
\mathbf{n}_{i,j} = \frac{\left[{\partial_x \phi_{i,j} }, {\partial_y \phi_{i,j}} \right]^{T}}{\left((\partial_x \phi_{i,j})^2+ (\partial_y \phi_{i,j})^2\right)^{1/2}} \quad \text{ , } \quad  \mathbf{\kappa}_{i,j} = \nabla \cdot \mathbf{n}_{i,j} = \frac{(\partial_{xx} \phi_{i,j})(\partial_{y} \phi_{i,j})^2-2(\partial_{x} \phi_{i,j})(\partial_{y} \phi_{i,j})(\partial_{xy} \phi_{i,j})+(\partial_{yy} \phi_{i,j})(\partial_{x} \phi_{i,j})^2}{\left((\partial_{y} \phi_{i,j})^2+(\partial_{x} \phi_{i,j})^2\right)^{3/2}}
$$
where $T$ denotes the transpose of a vector, and all partial derivatives are approximated by centered differences. One may use an improved, numerical approximation of $\mathbf{n}$ and $\mathbf{\kappa}$ suggested in the work \cite{macklin2006improved} to handle level set singularities and moving interfaces. In this work, we rely on the standard approximation as we deal with (piecewise) smooth, stationary domains. Accordingly, we compute the signed distance functions only once for a given spatial resolution, and such one-time computation is not too costly as they need to be calculated only near $\partial \Omega$.

\subsection{Ghost-point based finite difference approximation}
\label{sec:fd}
At every grid point $(x_i,y_j)\in \Omega$, we apply the standard 4th-order finite difference 
\begin{equation}
\label{stencil:fd}
	D_{h} u_{i,j} = \frac{-\frac{1}{12}u_{i-2,j}+\frac{4}{3}u_{i-1,j}-\frac{5}{2}u_{i,j}+\frac{4}{3}u_{i+1,j}-\frac{1}{12}u_{i+2,j}}{h^2}  + \frac{-\frac{1}{12}u_{i,j-2}+\frac{4}{3}u_{i,j-1}-\frac{5}{2}u_{i,j}+\frac{4}{3}u_{i,j+1}-\frac{1}{12}u_{i,j+2}}{h^2}
\end{equation}
to approximate $\Delta u$. All irregular and some regular points require sufficiently accurate ghost points in the first two ghost layers, for which we extrapolate irregular grid point values in the normal directions as follows. \textcolor{black}{Assume for now that $u$ satisfies the stationary nonlinear eigenvalue problem (\ref{eq:gpe}) with the homogeneous Dirichlet boundary condition}. Let $\mathbf{x}_g$ be a grid point in the ghost layer and apply the Taylor series expansion along the normal 
\begin{align}    
	u(\mathbf{x}_g) &= u(\mathbf{x}_g^{\star}) + \phi(\mathbf{x}_g)\partial_{\mathbf{n}(\mathbf{x}^{\star}_g)} u (\mathbf{x}_g^{\star}) + \frac{(\phi(\mathbf{x}_g))^2}{2}\partial_{\mathbf{n}(\mathbf{x}^{\star}_g)\mathbf{n}(\mathbf{x}^{\star}_g)} u (\mathbf{x}_g^{\star}) + O(h^3)  \\
	& = \phi(\mathbf{x}_g)\partial_{\mathbf{n}(\mathbf{x}^{\star}_g)} u (\mathbf{x}_g^{\star}) + \frac{(\phi(\mathbf{x}_g))^2}{2}\partial_{\mathbf{n}(\mathbf{x}^{\star}_g)\mathbf{n}(\mathbf{x}^{\star}_g)} u (\mathbf{x}_g^{\star}) + O(h^3)  
  \label{reconstr-Taylor}  
\end{align}
where $\mathbf{n}(\mathbf{x}^{\star}_g)$ is the local unit normal vector,  $\mathbf{x}^{\star}_g$ is the intersection point of the local normal line passing $\mathbf{x}_g$ with $\partial \Omega$ (the projection of $\mathbf{x}_g$ onto $\partial \Omega$).  In order to approximate the directional derivatives at $\partial\Omega$, we turn to the interior of $\Omega$ and apply a similar Taylor expansion  for an irregular point $\mathbf{x}$ to obtain
\[
\partial_{\mathbf{n}(\mathbf{x}^{\star})} u (\mathbf{x}^{\star}) =  \frac{u(\mathbf{x})}{\phi(\mathbf{x})} - \frac{\phi(\mathbf{x})}{2} \partial_{\mathbf{n}(\mathbf{x}^{\star})\mathbf{n}(\mathbf{x}^{\star})} u (\mathbf{x}^{\star}) + O(h^2)
\] 
where $\mathbf{x}^{\star}$ is the projection of $\mathbf{x}$ onto $\partial \Omega$.  We next approximate the second order normal derivatives by observing that 
$$\Delta u(\mathbf{x}^{\star}) = -2\lambda u(\mathbf{x}^\star)+2V(\mathbf{x}^{\star})u(\mathbf{x}^{\star})+2\beta|u(\mathbf{x}^\star)|^2u(\mathbf{x}^\star) = 0$$
 due to the homogeneous Dirichlet boundary condition. Using the Laplacian in polar coordinates
 $$\Delta u(\mathbf{x}^\star) =[\partial^2_{rr} u+\frac1r \partial_r u +\frac1{r^2} \partial^2_{\theta\theta} u ](\mathbf{x}^\star)~,
 $$
 where the origin of the polar coordinate system is set at the local center of curvature of the boundary, we have \cite{russell2003cartesian}
$$
\Delta u(\mathbf{x}^\star) = \partial_{\mathbf{n}(\mathbf{x}^{\star})\mathbf{n}(\mathbf{x}^{\star})} u (\mathbf{x}^{\star}) + \kappa(\mathbf{x}^\star)  \partial_{\mathbf{n}(\mathbf{x}^{\star})}u (\mathbf{x}^{\star}) 
$$
again due to the homogeneous Dirichlet boundary condition. Hence it follows that 
 \begin{equation}
  \partial_{\mathbf{n}(\mathbf{x}^{\star})}u (\mathbf{x}^{\star})   = \frac{1}{\phi(\mathbf{x})\left(1-\frac{1}{2}\kappa(\mathbf{x}^\star)\phi(\mathbf{x})\right)} {u(\mathbf{x})}  + O(h^2)~,
  \label{eq:grad}
\end{equation}
and consequently
\begin{equation}
\partial_{\mathbf{n}(\mathbf{x}^{\star})\mathbf{n}(\mathbf{x}^{\star})} u (\mathbf{x}^{\star}) = -\frac{{\kappa(\mathbf{x}^\star)}}{{\phi(\mathbf{x})\left(1-\frac{1}{2}\kappa(\mathbf{x}^\star)\phi(\mathbf{x})\right)} {} }u(\mathbf{x})+ O(h^2)~.
\label{eq:second_grad}
\end{equation}
Assuming a sufficiently smooth $\phi$, one can furnish $\kappa(\mathbf{x}^\star) = \kappa(\mathbf{x})- \phi(\mathbf{x})\nabla \kappa(\mathbf{x})\cdot \mathbf{n}(\mathbf{x}) + O(h^2)$, completing the approximation of normal derivatives at the boundary. As will be shown in the sequel, we will modify ${D}_h$ to take into account the effects of the {\color{black} ghost points that are completely determined by irregular grid points}.

Let us now turn to the case where the gradient flow is non-static, hence $\Delta u$ no longer vanishes at the boundary. A locally third-order extrapolation may still be constructed, but it is no longer needed since spatial resolution can be improved to higher accuracy \emph{once} we reach a steady state. Instead, we consider a more crude, yet simpler interpolation
 \begin{equation}
  \partial_{\mathbf{n}(\mathbf{x}^{\star})}u (\mathbf{x}^{\star})   = \frac{1}{\phi(\mathbf{x})} {u(\mathbf{x})}  + O(h)~,
  \label{eq:grad-2}
\end{equation} where $\mathbf{x}^\star$ denotes the projected boundary point as above. So far as the approximation of $\Delta$ is concerned, we resort to the standard second-order approximation 
\begin{equation*}
	\widehat{D}_{h} u_{i,j} = \frac{u_{i-1,j}-2u_{i,j}+u_{i+1,j}}{h^2}  + \frac{u_{i,j-1}-2u_{i,j}+u_{i,j+1}}{h^2}
\end{equation*}
as in the original BEFD. The resulting truncation error is $O(1)$ near the boundary, but we recall the result that linear extrapolation is sufficient to yield second order accurate solutions for linear Poisson problems \cite{gibou2013high}. In the current nonlinear setting, the semi-implicit treatment of the nonlinearity can only enhance the monotonicty of $\widehat{D}$, hence we expect no hindrance to the convergence. We set $\mathbf{D}_h$ to $\widehat{D}_h$ until an approximate steady state is reached, for instance,
$
 \max_{i,j} \frac{ |u_{i,j}^{n+1}-u_{i,j}^{n}|}{dt} < 10^{-8} 
$
to obtain an approximate ground state, which then is used to simulate the second phase of the gradient flow with $\mathbf{D}_h$ now chaged to ${D}_h$.

Since the Cartesian grid points do not conform with $\partial \Omega$, it is not true in general that for a fixed ghost point $\mathbf{x}_g$, there exists a corresponding irregular point $\mathbf{x}$ for which $\mathbf{x}^\star = \mathbf{x}^\star_g$. This issue can be addressed by solving an artificial transport equation to extend the necessary derivatives across the boundary along the normal directions \cite{lee2023ghost}. With our choice of the backward Euler, however, there arises the issue that numerical solutions to extended are unknown solutions at the next time step.

\subsection{PDE-based extension of nodal basis function}
\label{sec:enbf}
Given a function $u$ in $\Omega$, one can constantly extend it across $\partial \Omega$ \cite{fedkiw1999non} by solving 
\begin{equation}
\label{eq: extension}
\partial_t u + \mathbf{n}\cdot \nabla u = 0
\end{equation}
in $\Omega^c$, where $\mathbf{n} = {\nabla \phi}/|{\nabla \phi}|$ is the outward normal vector defined in Section \ref{sec:sd}. One can apply a first-order monotone scheme such as the first order upwinding method (which we use in this work) until a steady state is reached in the two ghost layers. For clarity of presentation, we assume for now that we seek such constant extension in lieu of more accurately extrapolated ghost point values. It is not the case that we can directly apply the constant extension technique with our choice of backward Euler time stepping, since the values of $u$ in $\Omega$ are unknown. However, it is evident that ghost values under the constant extension equation (\ref{eq: extension}) are linear combinations of the values at irregular grid points, hence we can write
\begin{equation*}
    V_{ghost}=A_h V_{irregular}~,
\end{equation*}
where $V_{ghost}$ is a column vector containing all ghost point values under the constant extension,  $V_{irregular}$ is a column vector of values at irregular grid points, and $A_h$ is an extension matrix to be determined. \\
\indent In order to compute the $k$-th column of $A_h$, we consider a nodal basis grid function $\eta_{k,h}$ given by
\begin{equation*}
\eta_{k,h} (i',j')  = 
\begin{cases}
1, \quad \text{if } i' = i, j' = j \\
0 \quad \text{ otherwise} 
\end{cases} 
\end{equation*}
where the index $(i,j)$ denotes the irregular grid point corresponding to the $k$-th entry of $U_{irregular}$. We then constantly extend $\eta_{k}$ and the resulting ghost values are then stored as the entries of $k$-th column of $A_h$ in the same ordering as the entries of $U_{ghost}$. We note that the columns of $A_h$ can be computed independently of one another and there are $O(1/h)$ number of basis functions to be extended. \textcolor{black}{Moreover, it takes $O(1)$ number of iterations to reach steady states in the two ghost layers of width $O(h)$. }The rows of the sparse matrix $A_h$ represent the interpolation weights that are automatically obtained without detailed examination of how then boundary curves intersects each grid cell. \\
\indent Now that the extension matrix $A_h$ is computed, we can use it to extend the first and second order local normal derivatives to the ghost layers with second order accuracy, without actually solving equation (\ref{eq: extension}). Combining formulas (\ref{reconstr-Taylor}), (\ref{eq:grad}) and (\ref{eq:second_grad}), we have at every ghost point
\begin{equation}
u(\mathbf{x}_g) = \phi(\mathbf{x}_g)\partial^h_{\mathbf{n}(\mathbf{x}^{\star}_g)} u (\mathbf{x}_g^{\star}) + \frac{(\phi(\mathbf{x}_g))^2}{2}\partial^h_{\mathbf{n}(\mathbf{x}^{\star}_g)\mathbf{n}(\mathbf{x}^{\star}_g)} u (\mathbf{x}_g^{\star}) + O(h^3)~,
\end{equation}
where 
$$
\partial^h_{\mathbf{n}(\mathbf{x}^{\star}_g)} u (\mathbf{x}_g^{\star})
$$
is the extended value of formula (\ref{eq:grad})
$$
\frac{1}{\phi(\mathbf{x})\left(1-\frac{1}{2}\kappa(\mathbf{x}^\star)\phi(\mathbf{x})\right)} {u(\mathbf{x})}
$$
stored at corresponding irregular grid points and
$$
\partial^h_{\mathbf{n}(\mathbf{x}^{\star}_g)\mathbf{n}(\mathbf{x}^{\star}_g)} u (\mathbf{x}_g^{\star})
$$
is the extended value of formula (\ref{eq:second_grad})
$$
-\frac{{\kappa(\mathbf{x}^\star)}}{{\phi(\mathbf{x})\left(1-\frac{1}{2}\kappa(\mathbf{x}^\star)\phi(\mathbf{x})\right)} {} }u(\mathbf{x})
$$
stored at corresponding irregular grid points. The above extensions of the first and second order normal derivatives are performed virtually through the extension matrix $A_h$.
Locally second order extrapolation via (\ref{eq:grad-2}) follows analogously. In light of the equations (\ref{eq:grad}) and (\ref{eq:second_grad}), we define the diagonal matrices
$$
C_h = \text{diag}(c_{k}), \qquad G_h = \text{diag}(g_{k}), \qquad \Phi_h = \text{diag}(\phi_{l})
$$
where $c_{k} =c_{i.j}= \kappa_{i,j} - \phi_{i,j} (\nabla \kappa)_{i,j} (\mathbf{n})_{i,j}$,
$g_k = ({\phi_{i,j}-\frac{1}{2}c_{i,j}(\phi_{i,j})^2})^{-1}$, and $\phi_l = \phi_{\hat{i},\hat{j}}$. Here, $(\nabla \kappa)_{i,j}$ is computed by the centered difference approximation, while $(i,j)$ and $(\hat{i},\hat{j})$ correspond to $k$-th and $l$-th entries of $U_{irregular}$ and $U_{ghost}$, respectively. We then arrive at the third-order extension mapping
\begin{equation}
    U_{ghost} = {\Phi_h\left(A_h-\frac{1}{2} \Phi_hA_hC_h\right) G_h} U_{irregular}, 
\end{equation}
where $U_{ghost}$ denotes a column vector containing all ghost point values, while $U_{irregular}$  represents a column vector of values at irregular grid points, arranged in the same order as described above. 

\textcolor{black}{The third order accuracy of our extension mapping rests on the assumption of sufficiently smooth boundary. Near sharp corners, one may impose a cutoff on the magnitudes of the local curvatures so that they are not larger than $O(\frac{1}{h})$  in light of our standard curvature computation in Section \ref{sec:sd}  on a grid of $O(h)$ resolution. However, our numerical tests are performed without limiting the curvature values, yet we observe no numerical artifacts, suggesting the robustness of our method. On a related note, we observe empirically that the diagonal dominance and sparsity of $D_h$ in (\ref{stencil:fd}) remains unchanged after it is modified by the extension mapping, which itself can be readily obtained thanks to its constitutive diagonal matrices. Let us also remark that should the resulting modification of $D_h$ become ill-conditioned (due to some irregular grid points being too close to the boundary), one may also apply the modified incomplete LU preconditioner \cite{gustafsson1978class} which is known to be the optimal choice for the Dirichlet boundary condition.}

 Our extrapolation approach differs from the existing one \cite{aslam2004partial} by decoupling the extension of the derviative information from the assembly of ghost point values. We solve the homogeneous PDE (\ref{eq: extension}) only once, instead of solving one homogeneous and two inhomogeneous equations (with different source terms). Our approach relies only on the normal derivatives at the boundary that are deliberately computed and stored at irregular grid points, in lieu of the normal derivatives \textcolor{black}{at both irregular and ghost grid points that are calculated and stored at those same grid points that they rightfully belong}. The simplicity of our approach is consistent with that of the original BEFD. However, a drawback of our nodal basis function extension is its dependence on a linear monotone scheme to compute the constant extension mapping, which in turn limits our ghost point extrapolation to be locally third-order accurate. \textcolor{black}{We leave it to future work to extend our approach to a higher-order scheme at the expense of greater computational complexity}.

\subsection{Level-set based numerical quadrature}
In order to enforce the $L^2$-normalization in the gradient flow, we need a sufficiently accurate numerical quadrature method that can handle the level set parametrization of the boundary curve.  Since quadratic extrapolation is shown to be third-order accurate in the linear Dirichlet-Laplace eigenvalue problems \cite{gibou2013high}, we seek a third-order accurate numerical quadrature for the $L^2$-norm by resorting to the cell-based quadrature method by Min and Gibou \cite{min2007geometric}. In their method, they construct a piecewise linear approximation of the boundary, from which the interface is divided into a disjoint union of simplices. They then apply the standard quadratures, namely the Grundmann-Moeller quadrature \cite{grundmann1978invariant} over the simplices,  and the trapezoidal rule over the interior rectangular cells (those that do not intersect the approximate boundary). Their method is second-order accurate, efficient, and robust against perturbation of the interface location.  We show that their geometric integration method is third-order accurate when it is applied to the square of a function which vanishes on the boundary.

\begin{lem}
    Suppose $u(x)$ is a smooth function on a smooth $\Omega$ with $u\vert_{\partial_\Omega} = 0$. Then, the Min and Gibou approximation of $\|u\|^2_{L^2(\Omega)}$ is third-order accurate.  
\end{lem}

\begin{proof}
    Let us write $C_{i,j} =  [x_i, x_{i+1}] \times  [y_j, y_{j+1}]$ and first consider the case $(i,j) \in B_h$, where $B_h = \{ (i,j) \in \mathbb{Z}^2: C_{i,j} \cap \partial \Omega  \neq \emptyset \}$.  Then, the Min and Gibou approximation of 
    $
    \int_{C_{i,j}\cap \partial \Omega} u^2 dx
    $ yields the local quadrature error $E_{i,j} = O(h^4)$  by construction, hence it is immediate to obtain
    $$
\sum_{ (i,j) \in B_h } {E_{i,j}} = O(h^3).    
    $$
    We now assume $({i,j}) \in I_h$ where $I_h = \{(i,j)\in \mathbb{Z}^2: C_{i,j} \subset \Omega \}$.  Then, simple calculation shows that 
    $$
        \int_{C_{i,j}} u^2 dx = \frac{h^2}{4}(u^2(i,j)+ u^2(i+1,j) + u^2(i,j+1) + u^2(i+1,j+1)) + E_{i,j}
    $$
    where the quadrature error $E_{i,j}$ is given by
    \begin{align*}
E_{i,j} = \frac{h^4}{12} \Delta (u^2) \bigg \vert_{(i^\star,j^\star)} & = \frac{h^4}{6}\left( u \Delta u + \nabla u \cdot \nabla u) \right)\bigg\vert_{(i^\star,j^\star)} = \frac{h^2}{6} \int_{C_{i,j}} \left( u \Delta u + \nabla u \cdot \nabla u \right) dx + O(h^6) = \frac{h^2}{6} \int_{\partial C_{i,j}} \left( u \mathbf{n} \cdot \nabla u  \right) dx + O(h^6). 
    \end{align*}
Here, $i^\star,j^\star$ denotes the center of cell $C_{i,j}$, and the last equality is obtained by the Green's identity. Since the fluxes across $\partial C_{i,j}$ cancel out when summing over all interior cells, except near $\partial \Omega$, we arrive at
$$ \sum_{(i,j) \in I_h} {E_{i,j}} = O(h^3) + \frac{h^2}{6} \sum_{(i,j) \in I_h} \int_{\partial C_{i,j}} (u \mathbf{n}\cdot\nabla u) dx  =  O(h^3) + \frac{h^2}{6}\int_{\partial \square  } (u \mathbf{n}\cdot\nabla u) dx =  O(h^3)
$$
where $\square = \bigcup_{(i,j) \in I_h} C_{i,j}$ and we have used $u = O(h)$ along $\partial \square$ due to the homogeneous Dirichlet boundary condition. 
\end{proof}

\textcolor{black}{We remark that the variational energy $\mathcal{E}_g$ can also be calculated to third-order accuracy due to its equivalence to $\mu_g - \frac{\beta}{2} \|u_g\|^4_{L^4(\Omega)}$.  }

\subsection{Applications to higher-order GP equations}
Our proposed BEFD method computes the ground states of the GP equation with the cubic linearity \ref{eq:gpe}, but it can also be readily applied to other variants of the equation that include higher order nonlinear terms. We take as the first example the cubic-quintic GP equation \cite{muryshev2002dynamics}
\begin{equation} \label{eq:cqgpe}
    \mu u(x) = -\frac{1}{2} \Delta u(x) + V(x) u(x) + \beta |u(x)|^2u(x) + \gamma |u(x)|^4u(x)
\end{equation}
suggested as a better model for the BEC in a long thin cylindrical geometry. Here, the coefficients $\beta$ and $\gamma$ denote the strengths of two- and three-particle interactions, respectively. We assume for simplicity that both $\beta$ and $\gamma$ are positive (repulsive). 
We adapt our BEFD algorithm by treating the quintic nonlinearity semi-implicitly as $|u^n_{i,j}|^4\tilde{u}^{n+1}_{i,j} $. There have been studies on particular solutions to the (purely) quintic and cubic-quintic equations, as well as a more systematic investigation in radial domains \cite{luckins2018bose}. Our level set based approach may be utilized for further studies in more general domain shapes. 

As our next example, we consider the modified GP equation \cite{veksler2014simple} 
\begin{equation} \label{eq:hoigpe}
 \mu u(x) = -\frac{1}{2} \Delta u(x) + V(x) u(x) + \beta |u(x)|^2u(x) - \delta \Delta (|u(x)|^2) u(x)
\end{equation}
which accounts for the higher order interaction correction to the binary interaction. Here, $\delta$ is assumed to be non-negative to guarantee the unique positive ground state  \cite{bao2019ground}. We follow the approach proposed in the work \cite{ruan2018normalized} by applying the BEFD with convex-concave splitting of the higher-order interaction term. The splitting scheme is proposed as an alternative to the original BEFD (without the splitting) since the latter suffers from the severe stability issue in the presence of the higher order term. Specifically, we consider the semi-discrete gradient flow
\begin{align*}
\frac{\tilde{u}_{i,j}^{n+1}-u_{i,j}^{n}}{\Delta t} &= \left(\frac{1}{2} + 2\delta |u^n_{i,j}|^2\right) \Delta {\tilde{u}^{n+1}_{i,j}} - V_{i,j}\tilde{u}^{n+1}_{i,j} - \beta |u^n_{i,j}|^2 \tilde{u}^{n+1}_{i,j} + 2\delta |(\nabla u)^n_{i,j}|^2 u^n_{i,j} \\
u^{n+1} &= \frac{\tilde{u}^{n+1}}{\|\tilde{u}^{n+1}\|_2}
\end{align*}
where the gradient $\nabla u$ is approximated by the finite difference operator $\mathbf{G}_h$ which takes the form of
\begin{align}
\widehat{G}_h  u_{i,j}  &=  \left[\frac{u_{i+1,j} - u_{i-,j}}{2h}, \frac{u_{i,j+1} -   u_{i,j-1} }{2h} \right ], \\
{G}_h  u_{i,j}  &=   \left[\frac{u_{i-2,j} -  8 u_{i-1,j} + 8 u_{i+1,j} - u_{i+2,j}}{12h}, \frac{u_{i,j-2} -  8 u_{i,j-1} + 8 u_{i,j+1} - u_{i,j+2}}{12h} \right ]
\end{align}
when $\widehat{D}_h$ and $D_h$ are used, respectively. The fully discrete scheme is 
\begin{align}
\frac{\tilde{u}_{i,j}^{n+1}-u_{i,j}^{n}}{\tilde\Delta t} &= \left(\frac{1}{2} + 2\delta |u^n_{i,j}|^2\right) \mathbf{D}_h {\tilde{u}^{n+1}_{i,j}} - V_{i,j}\tilde{u}^{n+1}_{i,j} - \beta |u^n_{i,j}|^2 \tilde{u}^{n+1}_{i,j} + 2\delta |\mathbf{G}_h u^n_{i,j}|^2 u^n_{i,j} \\
u^{n+1} &= \frac{\tilde{u}^{n+1}}{\|\tilde{u}^{n+1}\|_2}.
\end{align}
We \textcolor{black}{modify $\widehat{G}_h$ and ${G}_h$ in the same fashion as $\widehat{D}_h$ and ${D}_h$ in Section \ref{sec:enbf}}; the centered approximations of the gradient are used to match the corresponding order of approximations for $\mathbf{D}_h$. The original work \cite{ruan2018normalized} shows that the BEFD with the splitting is conditionally stable, yet it allows a much larger time step than the BEFD without any splitting. Another alternative method is proposed in \cite{bao2019computing} which adopts the density function formulation and applies the accelerated projected gradient method to solve a convex minimization problem.

\section{Numerical Experiments} \label{sec:ne}
In this section, we demonstrate the performance of our proposed level-set based BEFD method highlighting its accuracy for various domain shapes and potentials. We use exact signed distance functions and set the time step size $\Delta t = h$ unless otherwise stated. The eigenvalues and eigenfunctions of the linear problems are computed using the built-in command \textsc{Matlab} \textsf{ eigs}.  We estimate the convergence rates by comparison with the reference solutions computed on a very refined grid. 

\subsection{Laplace-Dirichlet eigenvalue on L-shaped domain}
As our preliminary test, we repeat the first numerical test in \cite{heid2021gradient} by solving 
\begin{eqnarray*}
	-\Delta u &= \mu u \quad & \text{ in } \Omega\\
	u  &= 0 \quad & \text{ in } \partial \Omega
\end{eqnarray*}
where  $\Omega = (-1,1)^2 \backslash [0,1]^2$. We compute the smallest eigenvalue $\mu_g$ using a larger computational domain $[-2\pi/5,2\pi/5]$ with mesh sizes $h = \frac{4\pi}{300},\frac{4\pi}{400},\frac{4\pi}{500}, \frac{4\pi}{600}, \frac{4\pi}{700}$ such that the Cartesian grids and the corners of $\Omega$ need not coincide. We compare the computed values with the reference value $ \mu_g \approx 9.639723844021$ provided in \cite{betcke2005reviving}. The singularity of the eigenfunction near the re-entrant corner leads to the expected convergence rate of $O\left(h^{\frac{4}{3}}\right)$, which is consistent with our obtained convergence rate of 1.3269. 

\subsection{Harmonic-Lattice potential in domains with constant curvature}
Our next numerical test considers the fully nonlinear GP equation (\ref{eq:nls}) with the harmonic potential $V(x,y) = (x^2+y^2)/2$ and $\beta = 50$. We use the square domain of the side length $2$ embedded in the computational domain $[-5\pi/6,-5\pi/6]^2$. We choose the mesh sizes $h = \frac{5\pi}{180},\frac{5\pi}{240}, \dots \frac{5\pi}{300},\frac{5\pi}{360}$, and our numerical solutions are compared with the reference value $\mu_g \approx  6.188543396102850$ provided in \cite{engstrom2022higher}. We solve both the original and the rescaled equations (\ref{eq:sgp}), and the results are presented in Figure \ref{fig:sh}.

	\begin{figure}[H]
	\centering
	\subfloat[]{\includegraphics[scale=0.33]{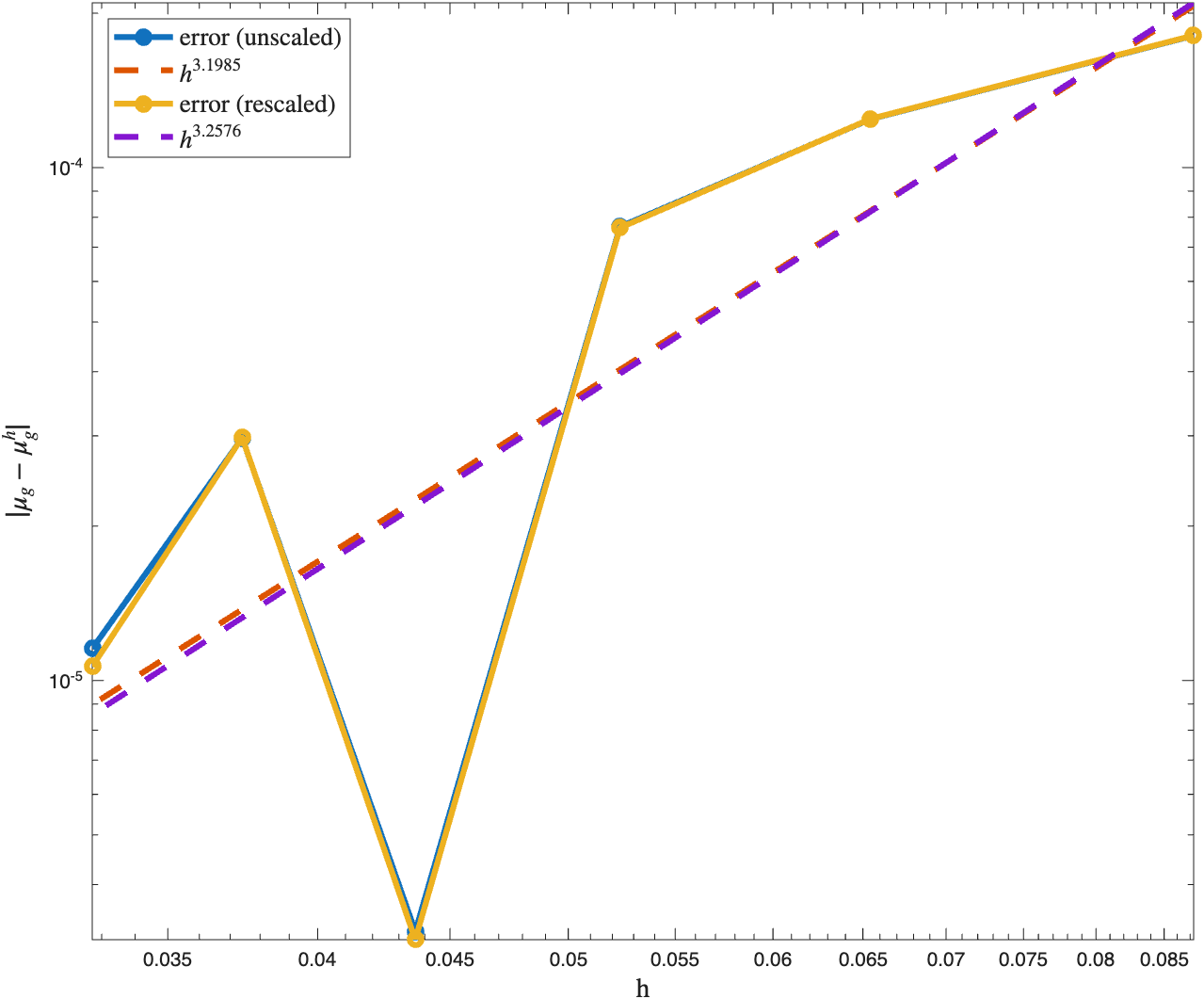}}
        \hspace{8mm}    
	\subfloat[]{\includegraphics[scale=0.35]{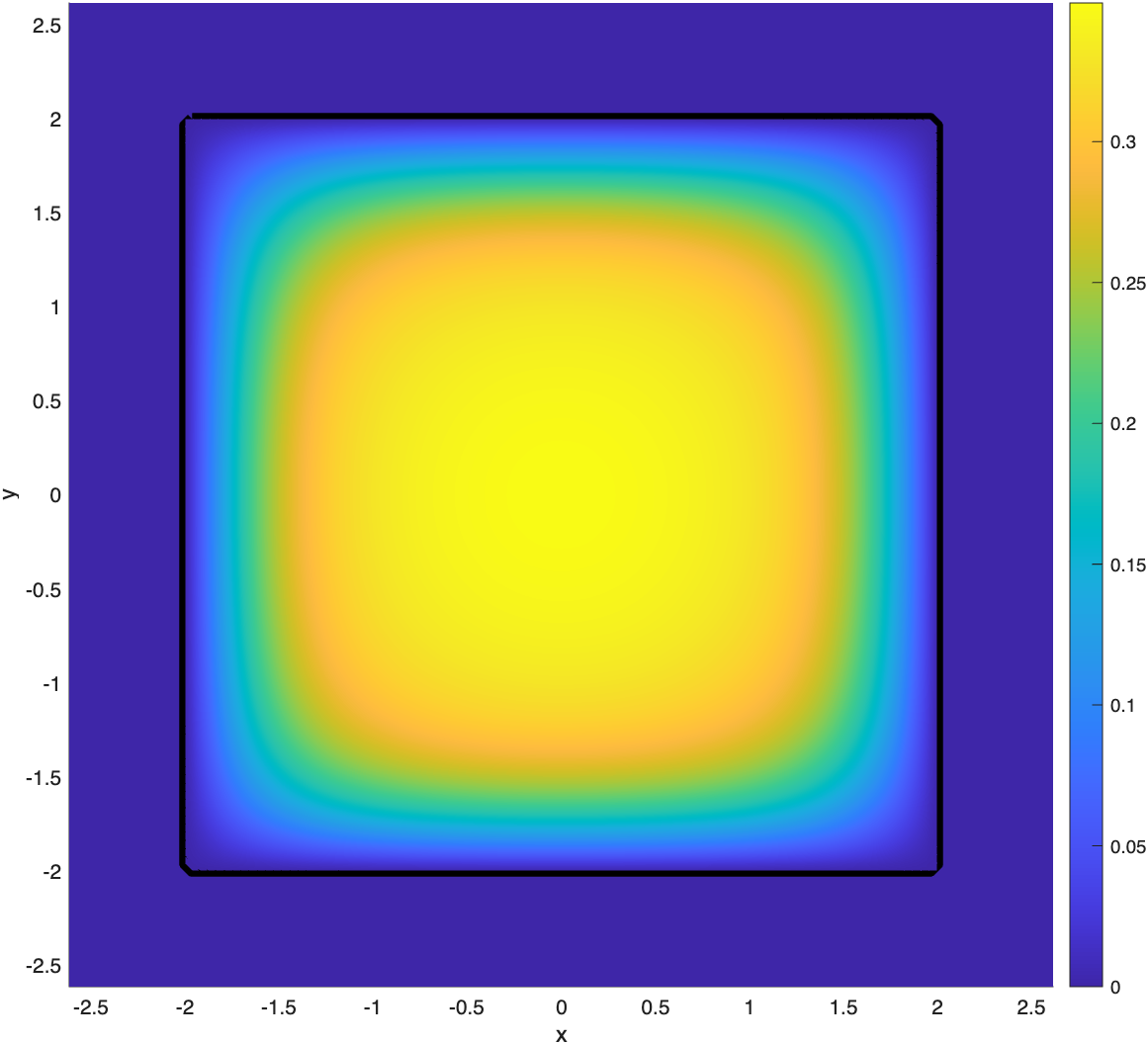}}
	\caption{Convergence plot for $\mu_g$ (a) and solution snapshot (b) in the square domain with $\beta = 50$ and harmonic potential}
	\label{fig:sh}
\end{figure}

Next, we consider the superposition of the harmonic potential plus the optical lattice
$
V(x,y) = \frac{1}{2}\left(x^2+y^2\right) + 50 \left(\sin^2 \left({\pi x}\right) + \sin^2\left({\pi y}\right)\right) 
$
in a manner similar to that of \cite{zhang2016krylov}. We use $\beta = 200$, and choose the computational domain $[-\pi,\pi]^2$ to enclose the circular domain $\Omega$ centered at $(0,0)$ with the radius of $2$. 
We obtain $\mu_g = 68.0881$ with the estimated convergence rate of $ 2.7097$ (Figure \ref{fig:ch}). We also report the convergence of the energy to the ground state energy value of $52.8319$ as well as the energy evolution in Figure \ref{fig:ch_e} to further validate our method. 

	\begin{figure}[H]
	\centering
	\subfloat[]{\includegraphics[scale=0.33]{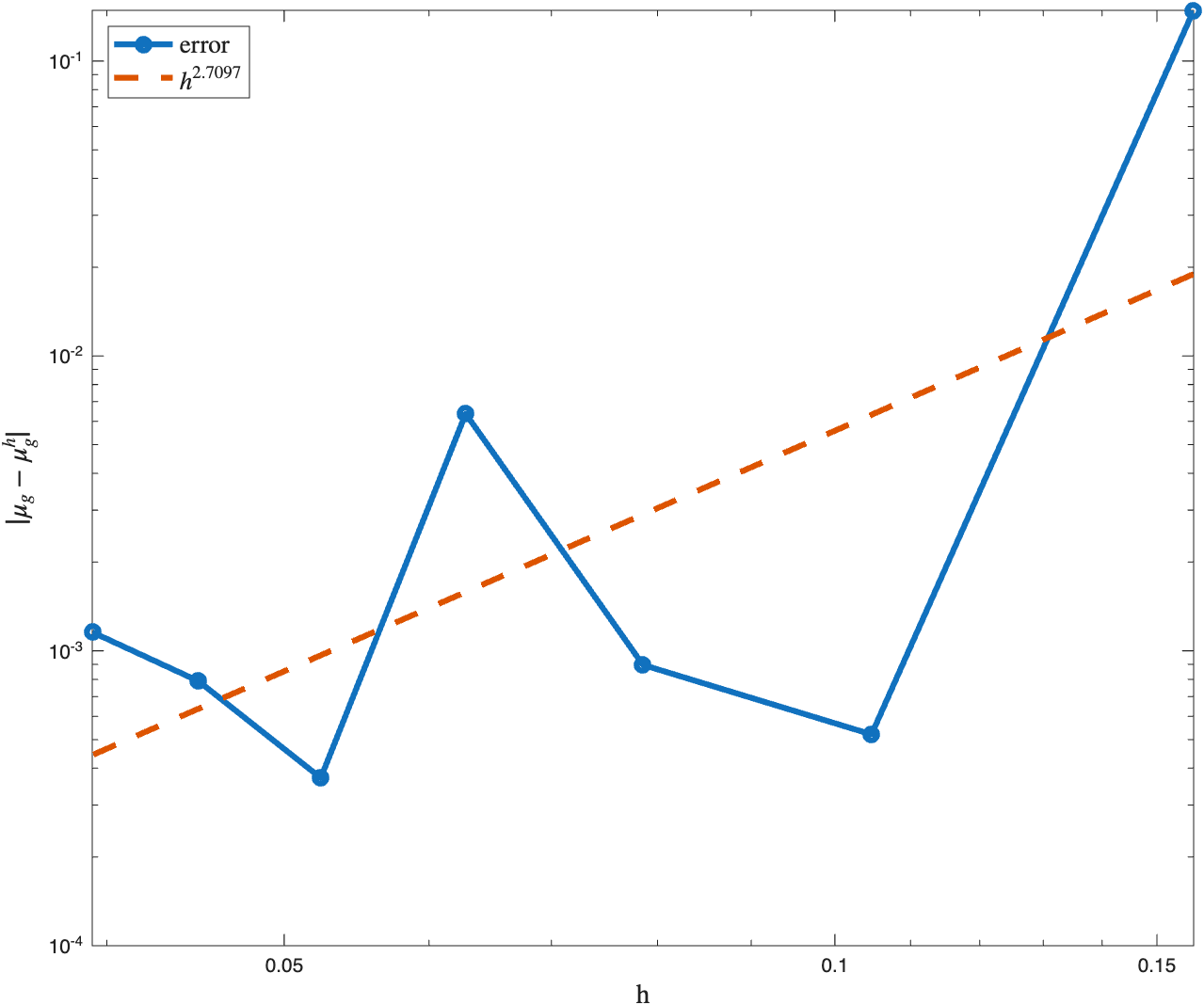}}
    \hspace{8mm} 
	\subfloat[]{\includegraphics[scale=0.35]{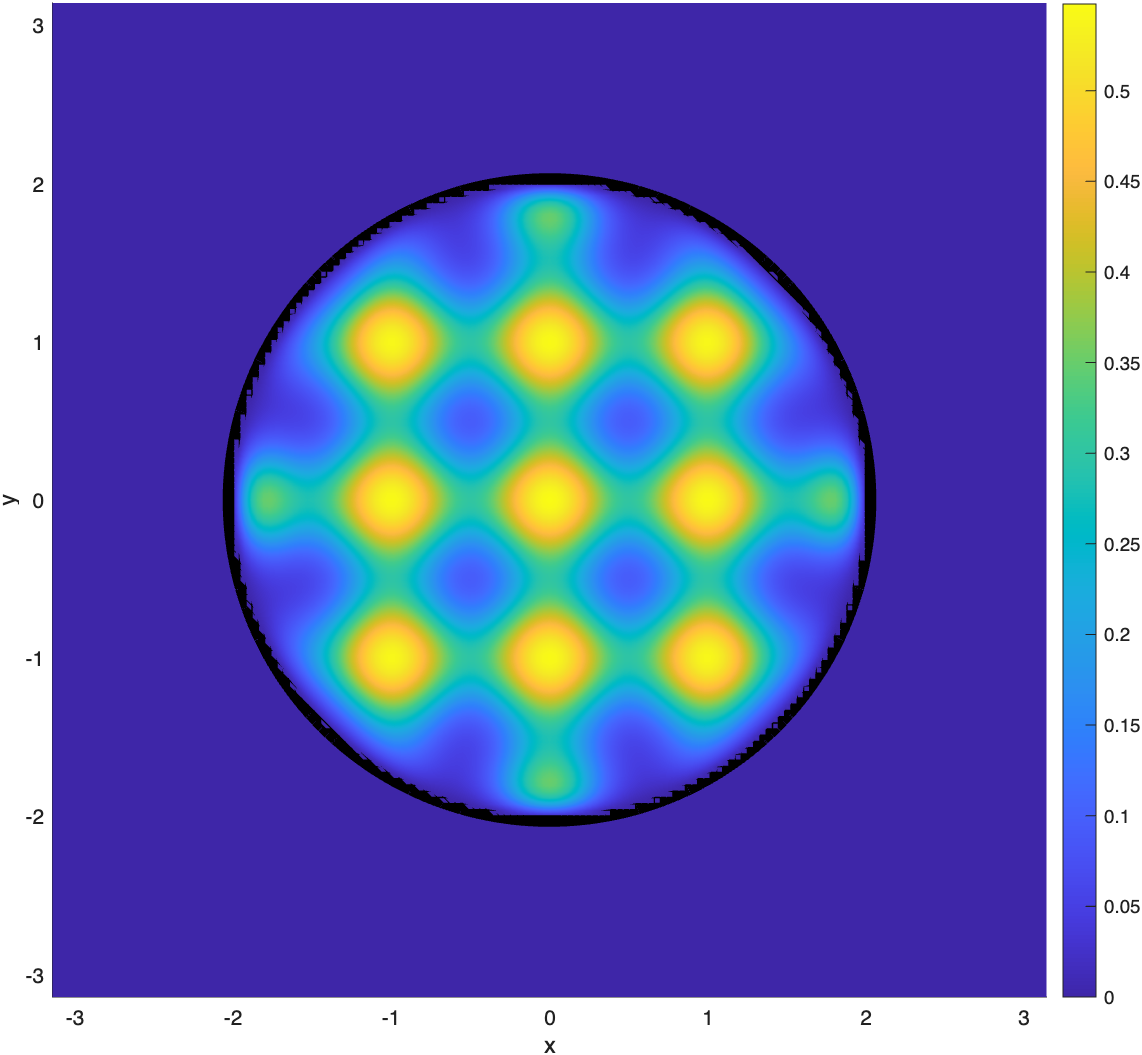}}
	\caption{Convergence plot for $\mu_g$ (a) and the solution snapshot (b) in the circular domain with $\beta = 200$ and harmonic-lattice potential.}
	\label{fig:ch}
\end{figure}

	\begin{figure}[H]
	\centering
	\subfloat[]{\includegraphics[scale=0.33]{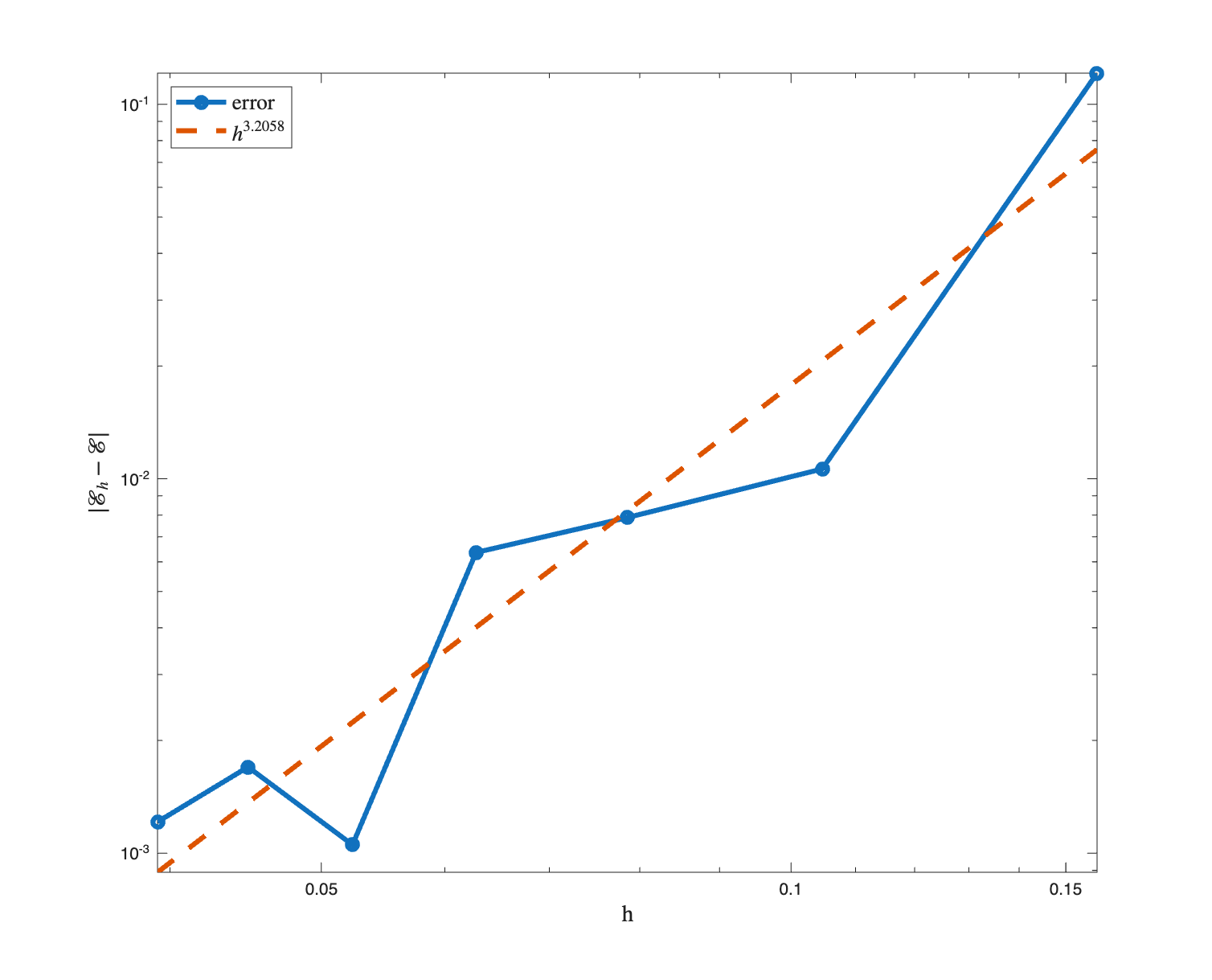}}
	\subfloat[]{\includegraphics[scale=0.33]{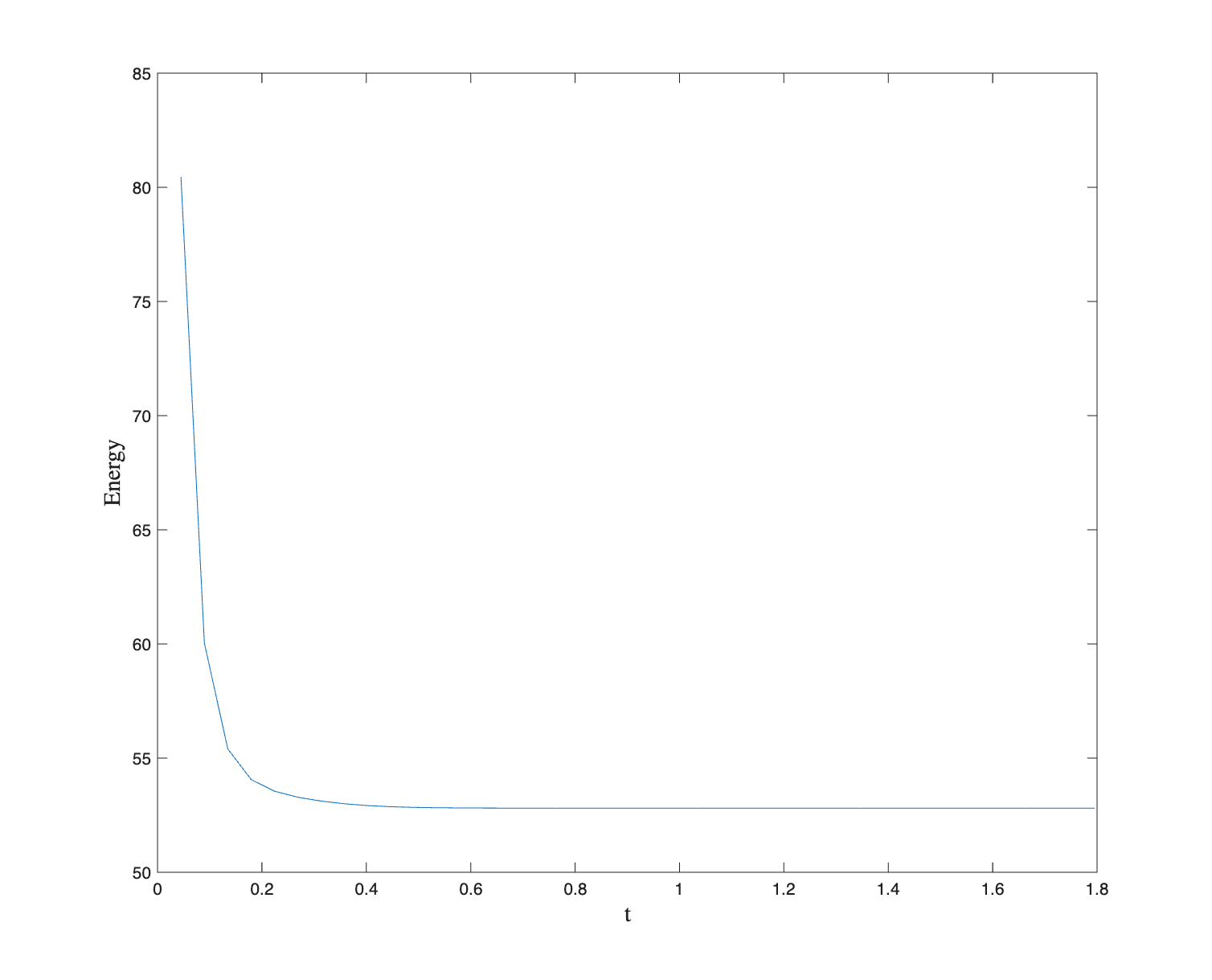}}
	\caption{Convergence plot for energy (a) and energy evolution plot (b) in the circular domain with $\beta = 200$ and harmonic-lattice potential.}
	\label{fig:ch_e}
\end{figure}

\subsection{Box potential in elliptical domain}
\label{subsec:be}
We choose the elliptical domain $\Omega = \left\{(x,y)\in \mathbb{R}^2: \frac{x^2}{\alpha^2} + \frac{y^2}{\gamma^2} = 1\right\}$ with $\alpha = 1.5 $, $\gamma = 2.0$ embedded in our computational domain $[-\pi,\pi]^2$. In \cite{salasnich2022bose}, BECs are confined in an elliptical waveguide by means of a quantum-curvature potential. We instead consider geometric confinement by setting set $\beta = 4, V(x,y) \equiv 0$ (i.e. box potential for $\Omega$). We obtain $\mu_g = 1.8055$ with the estimated convergence rate of $3.1819$. In addition to the ground state, we compute the first excited state and the corresponding chemical potential value of $ 3.0755$. We provide an illustration of the interaction between the geometric elliptic confinement and an ellipse-shaped potential \cite{adhikari2012dipolar} $V(x,y) = 4\left( \frac{x^2}{\gamma^2} + \frac{y^2}{\alpha^2}-0.3\right)^2$ for which the ground state chemical potential is $2.3411$. The results are presented in Figure \ref{fig:ebec}.

	\begin{figure}[H]
	\centering
	\subfloat[]{\includegraphics[scale=0.32]{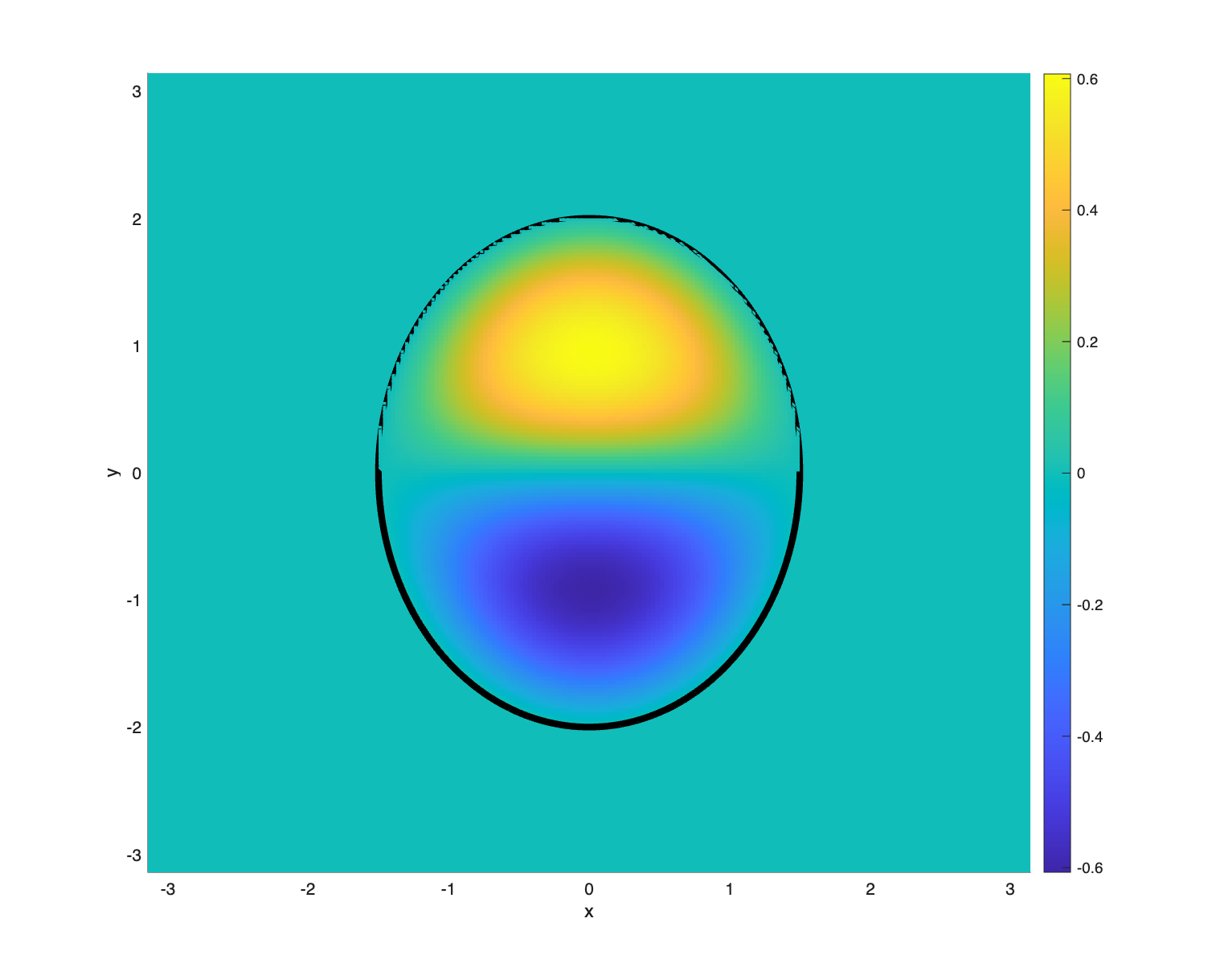}}
	\subfloat[]{\includegraphics[scale=0.35]{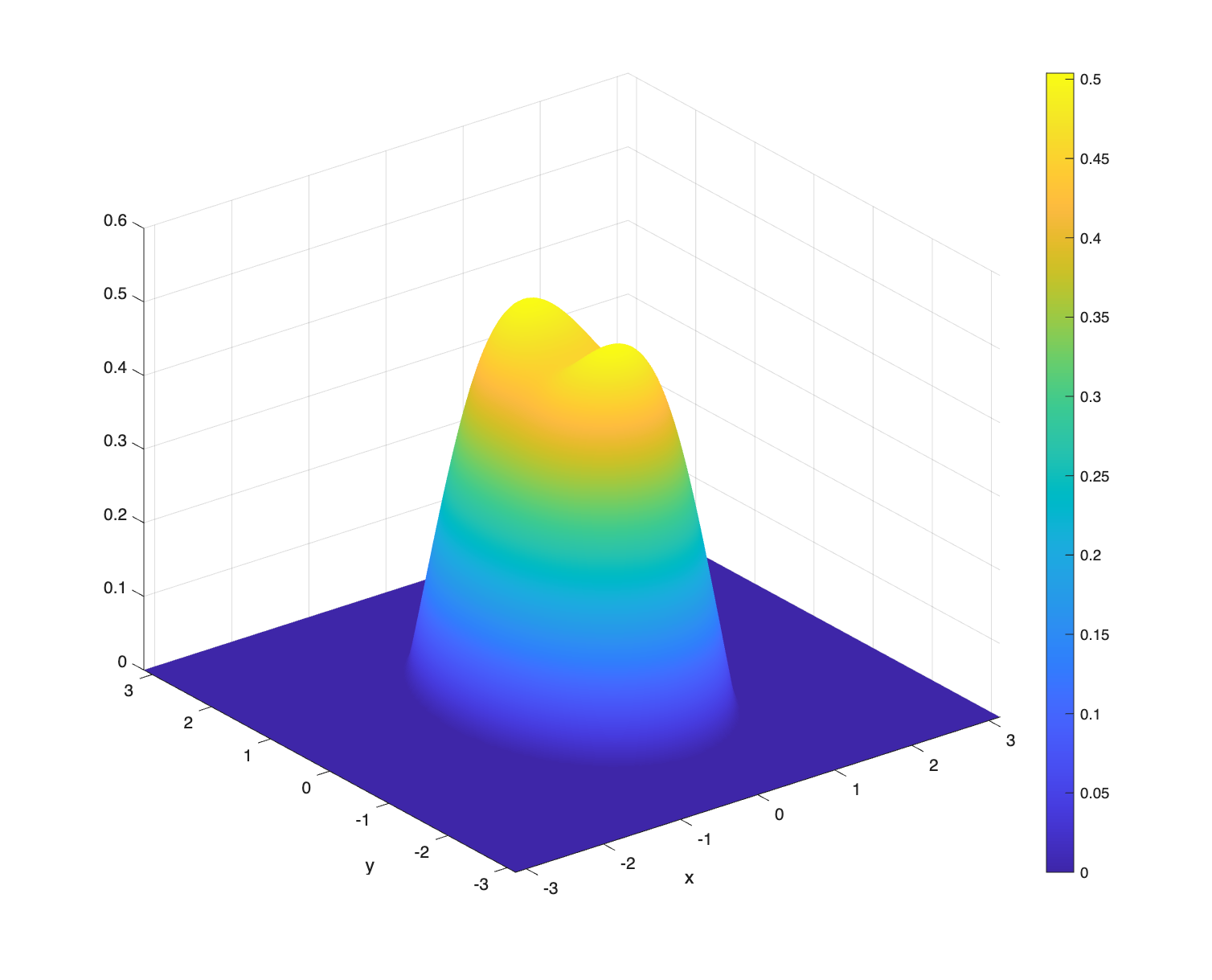}}
	\caption{Solution snapshots with $\beta = 4$: (a) first excited state subject to box potential; (b) ground state subject to ellipse-shaped potential.}
	\label{fig:ebec}
\end{figure}

\subsection{Obstacle potential in crescent-shaped domain}
We consider a moon-shaped domain in light of the growing interest in crescent-shaped waveguides \cite{golkebiewski2023spin} as well as the potential utility of convex domains to model optical tweezers for BEC clouds \cite{van2020perturbation}. We use a Gaussian obstacle potential  \cite{kwak2023minimum} by setting $V(x,y) = 4e^{(-2\left(x+0.35\right)^2-y^2 )}$ in the computational domain of $\left[{-\pi}/{3},{-\pi}/{3}\right]^2$. With $\beta = 10$, we obtain $\mu_g = 86.5431$ with the estimated convergence rate of $3.2460$; the potential peak is marked by $\diamondsuit$ in Figure \ref{fig:moon}. As in the previous subsection \ref{subsec:be}, we attempt to compute the excited states by using the excited states of the corresponding linear problem as the initial data. However, we observe that even with a small value of $\beta$ (e.g. $\beta = 0.001$), our method does not converge to the excited states but  to the ground states.

	\begin{figure}[H]
	\centering
	\subfloat[]{\includegraphics[scale=0.33]{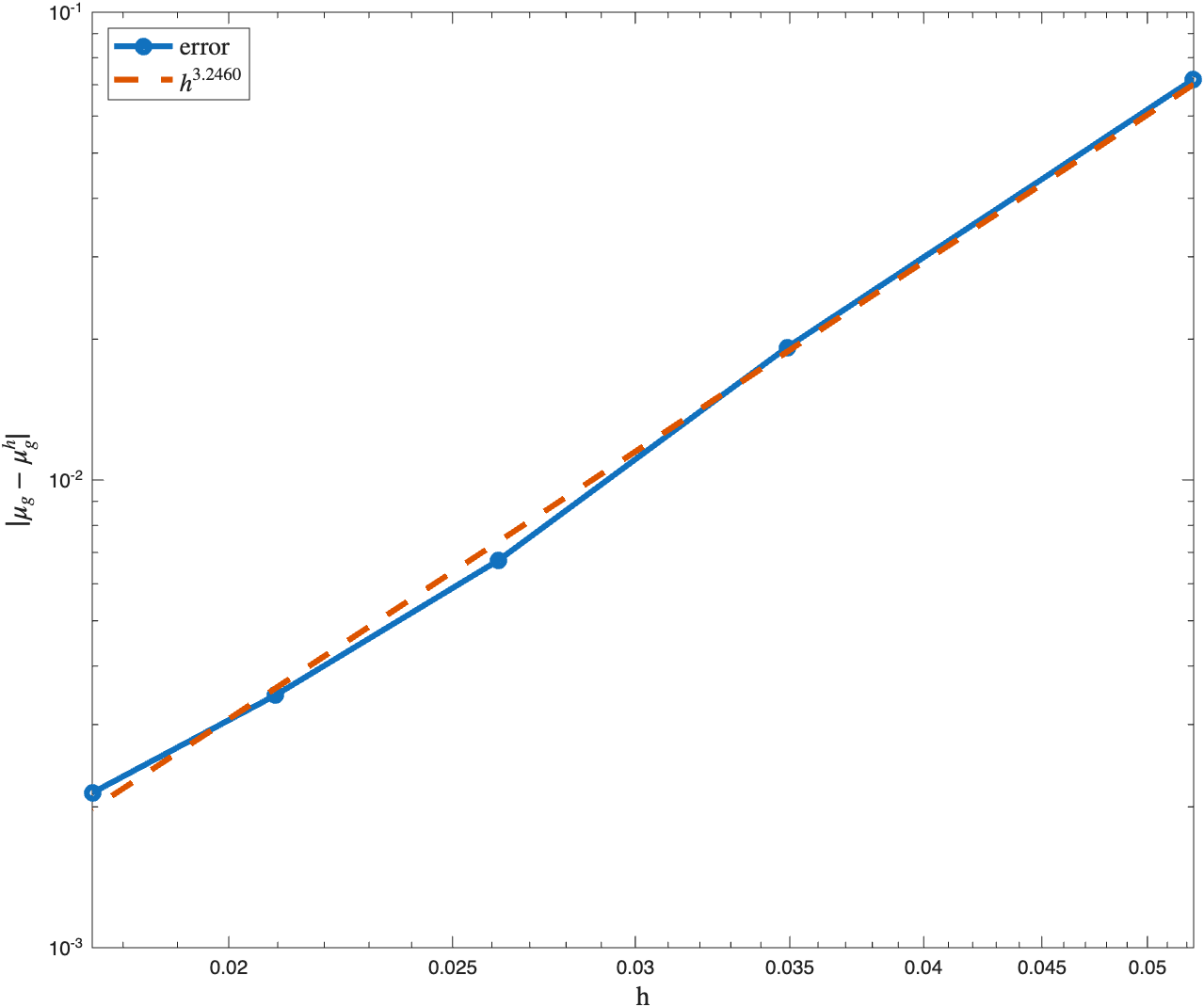}}
        \hspace{8mm} 
	\subfloat[]{\includegraphics[scale=0.33]{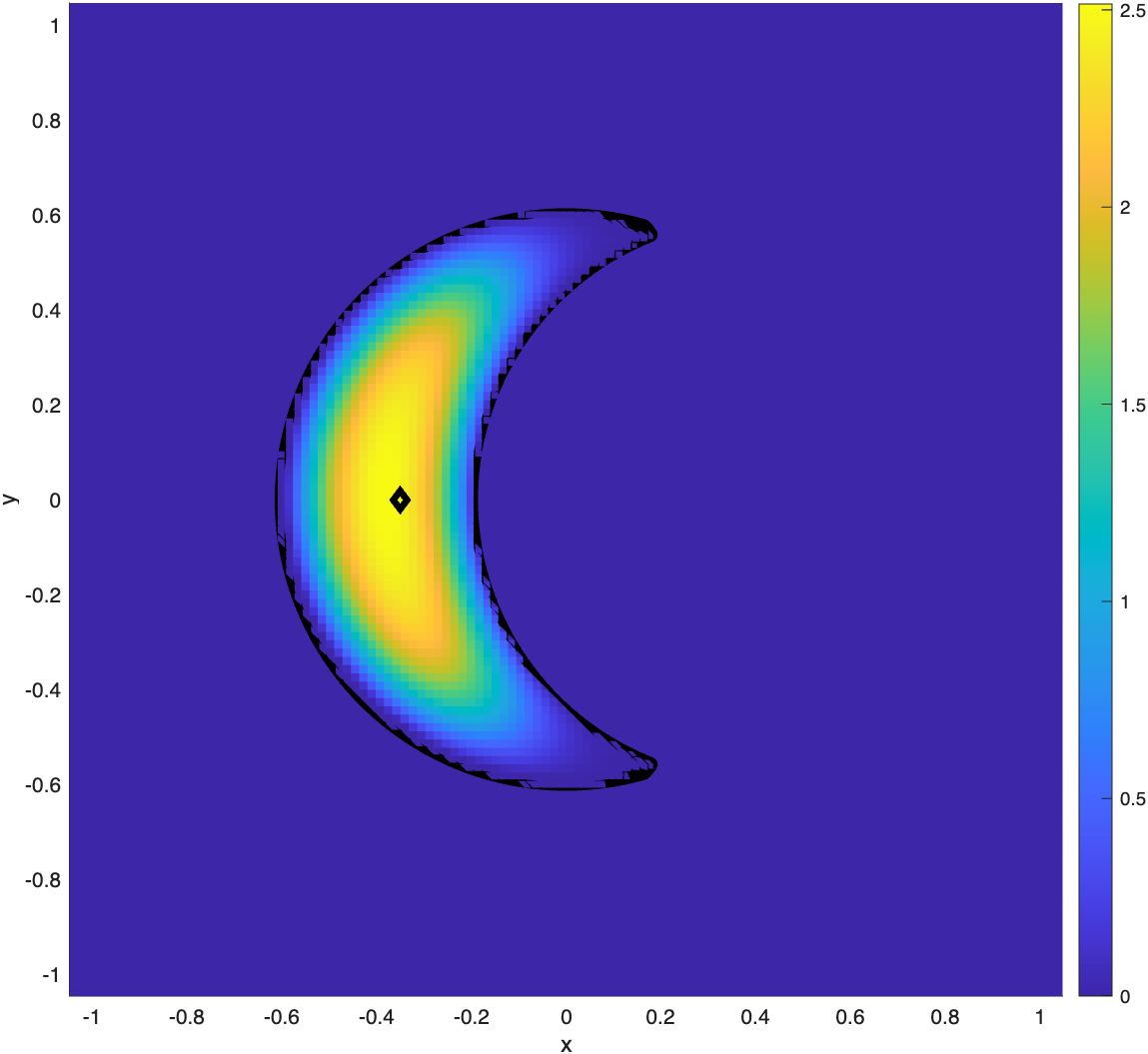}}
	\caption{Convergence plot for $\mu_g$ (a) and solution snapshot (b) in the moon-shaped domain with $\beta = 10$ subject to the Gaussian potential.}
	\label{fig:moon}
\end{figure} 

\subsection{Cubic-quintic interactions in a circular sector}
We solve the cubic-quintic GP equation (\ref{eq:cqgpe}) with $\beta = \gamma = 1$. We choose the quantum pendulum potential $V(x,y) = 1-\cos\left({2\pi}  \sqrt{x^2+y^2}\right)$ which models an optical lattice trap  \cite{luckins2018bose}. Our calculation yields $\mu_g = 2.432$, and due to geometric confinement, we note that the solution does not diverge, in contrast to the case noted in \cite{luckins2018bose}. We also consider the purely quintic case of $\beta =0, \gamma = 1$, but the results are omitted here, since they are qualitatively similar to the cubic quintic case presented in Figure \ref{fig:pie}.

	\begin{figure}[]
	\centering
	\subfloat[]{\includegraphics[scale=0.33]{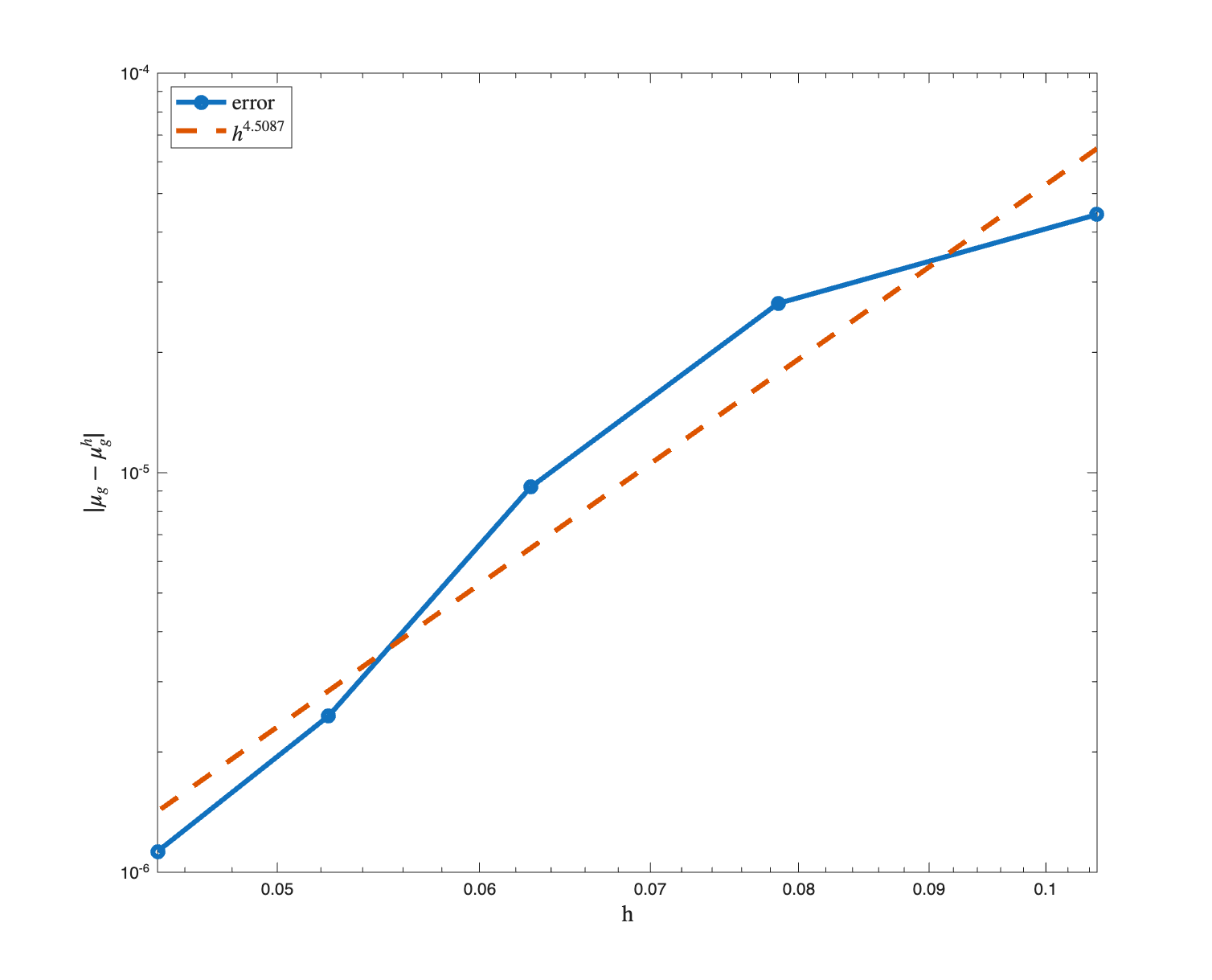}}
	\subfloat[]{\includegraphics[scale=0.33]{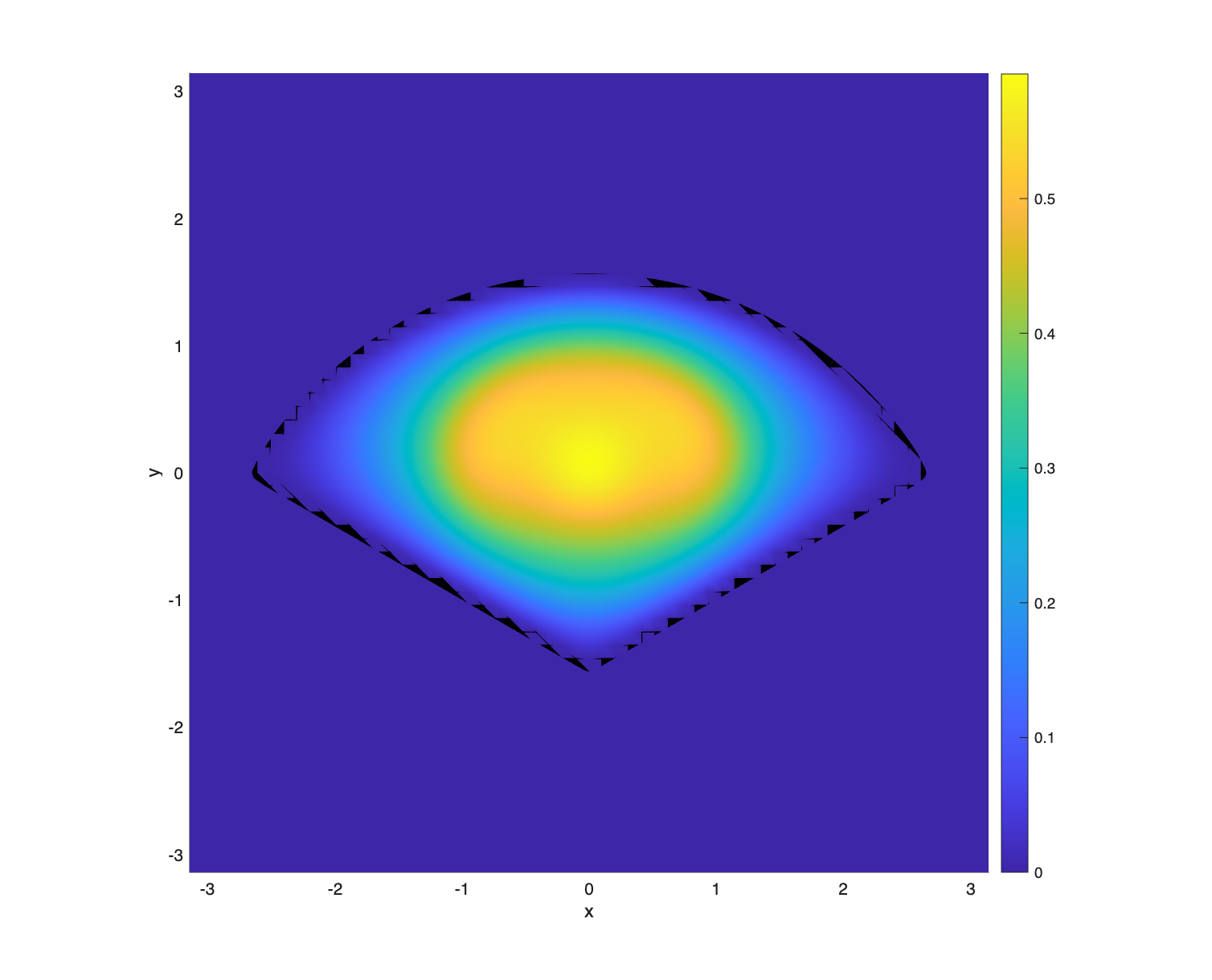}}
	\caption{Convergence plot for $\mu_g$ (a) and solution snapshot (b) in the circular sector with $\beta = \gamma = 1$ subject to the quantum pendulum potential.}
	\label{fig:pie}
\end{figure} 

\subsection{Modified GPE with the higher order interaction term}

In our last experiment, we revisit the elliptical domain and solve the modified GP equation with the higher-order interaction term (\ref{eq:hoigpe}). We take $\beta = \delta = 10$ with $V$ set to the box potential of the elliptical domain. Following the work \cite{ruan2018normalized}, we choose a small, fixed time step size $k = 0.001$ to illustrate the spatial accuracy of our method (Figure \ref{fig:e_hoi}). We obtain $\mu_g = 6.1360$ and the corresponding ground state energy of $5.7685$.

	\begin{figure}[H]
	\centering
	\subfloat[]{\includegraphics[scale=0.33]{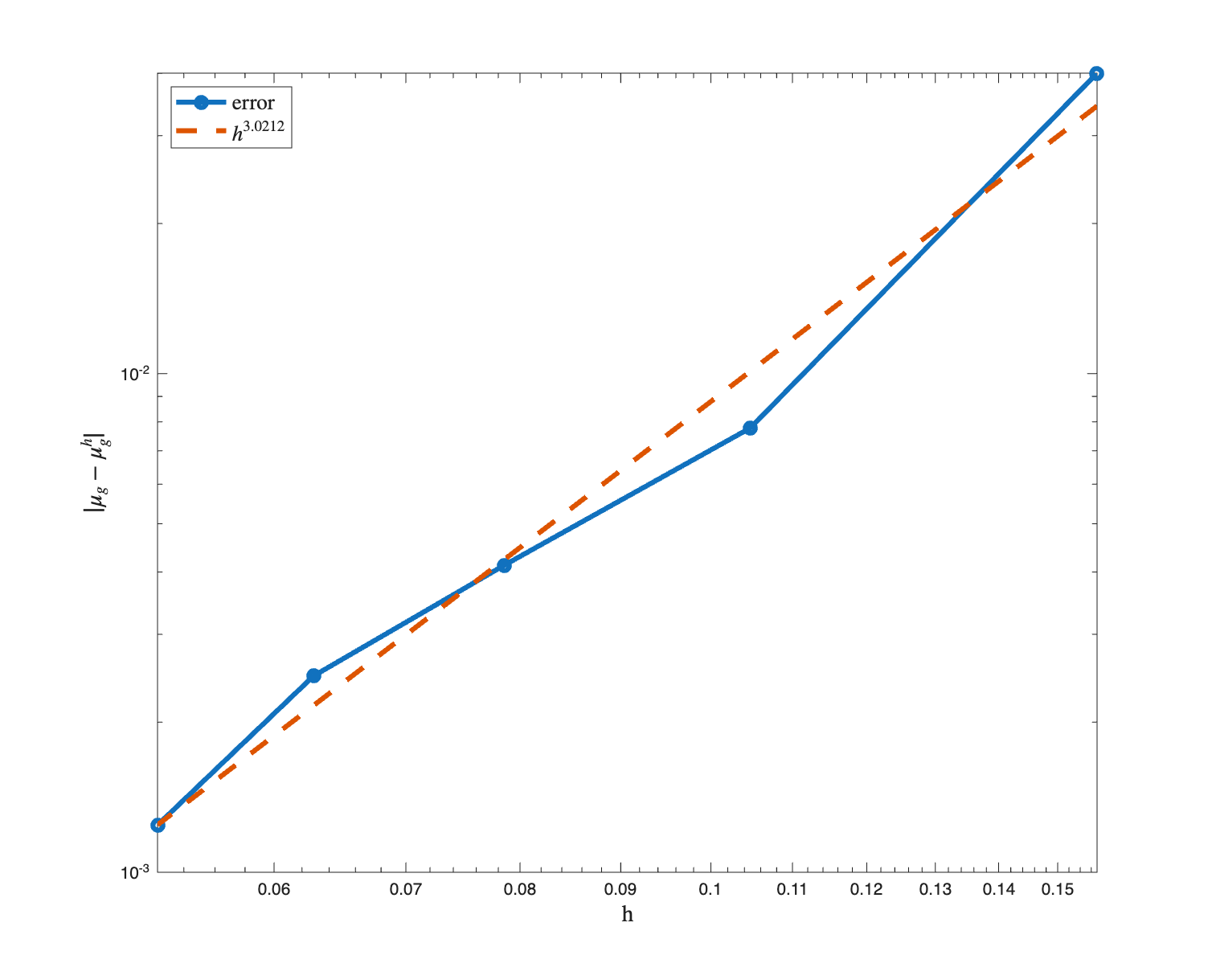}}
	\subfloat[]{\includegraphics[scale=0.33]{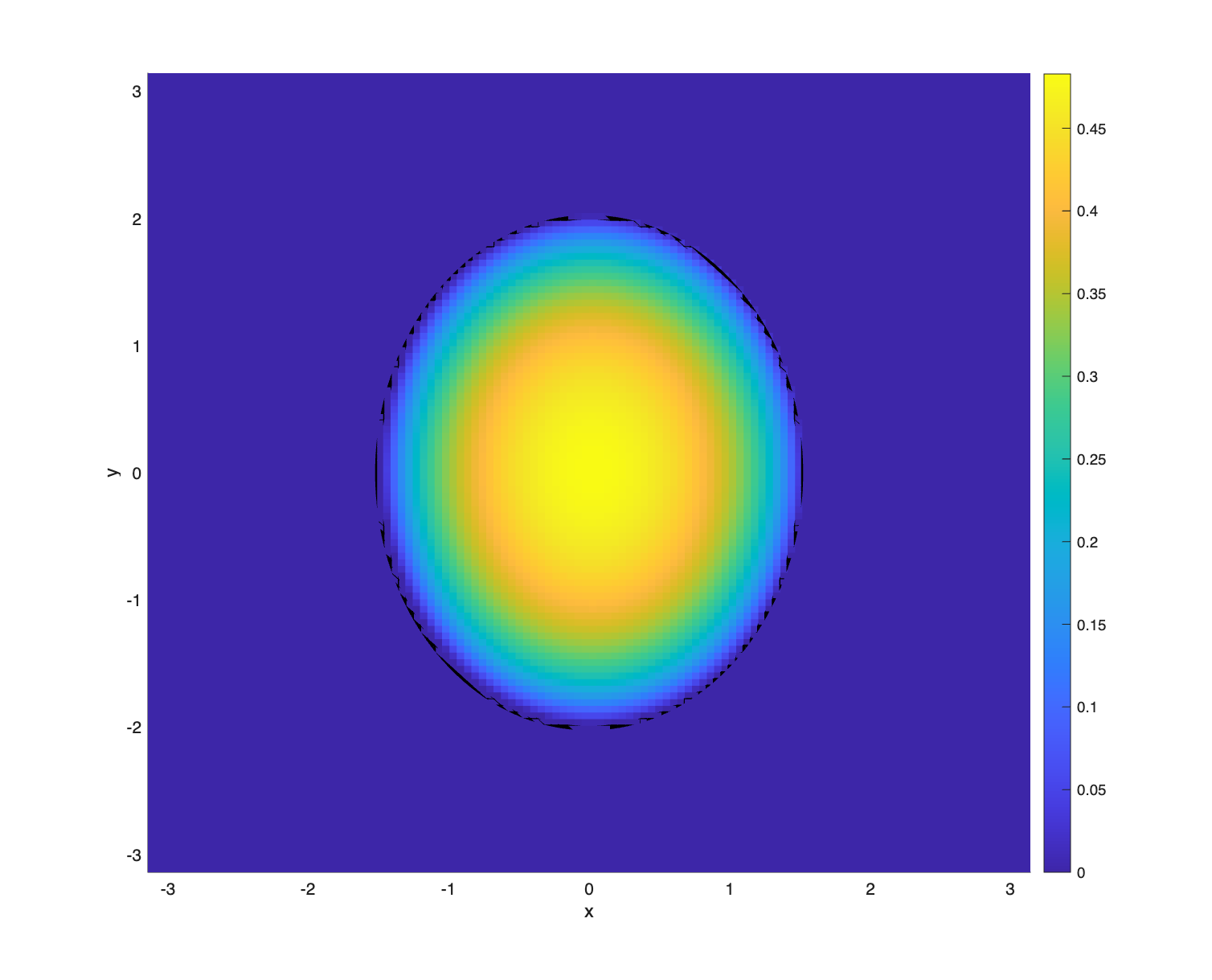}}
	\caption{Convergence plot for $\mu_g$ (a) and solution snapshot (b) in the elliptical domain with $\beta = \delta = 10$ and the box potential.}
	\label{fig:e_hoi}
\end{figure}

\section{Conclusion}  \label{sec:conc}
In this paper, we have proposed the level-set based BEFD method to compute the ground states of BEC in bounded domains with curved boundary. At its core, our method is a ghost point method in which the interpolation weights for the ghost points are automatically and explicitly computed by constant extension of nodal basis functions. Our numerical experiments indicate that the proposed method is third-order accurate, without compromising the convergence order from the linear setting. The simplicity and flexibility of our method make it a viable tool for further numerical studies of the BECs to address relevant, important questions of geometric nature.

\section*{CRediT authorship contribution statement}
\textbf{Hwi Lee}: Conceptualization, Methodology, Software, Writing - Original Draft. 
\textbf{Yingjie Liu}: Methodology, Writing - Review $\&$ Editing.

\section*{Declaration of interests}
The authors declare that they have no known competing financial interests or personal relationships that could have appeared to influence the work reported in this paper.

\section*{Data availability}
No data was used for the research described in the article.

\bibliographystyle{elsarticle-num} 
\bibliography{references}

\begin{thebibliography}{10}
\expandafter\ifx\csname url\endcsname\relax
  \def\url#1{\texttt{#1}}\fi
\expandafter\ifx\csname urlprefix\endcsname\relax\def\urlprefix{URL }\fi
\expandafter\ifx\csname href\endcsname\relax
  \def\href#1#2{#2} \def\path#1{#1}\fi

\bibitem{anderson1995observation}
M.~H. Anderson, J.~R. Ensher, M.~R. Matthews, C.~E. Wieman, E.~A. Cornell,
  Observation of bose-einstein condensation in a dilute atomic vapor, science
  269~(5221) (1995) 198--201.

\bibitem{bose1924plancks}
Bose, Plancks gesetz und lichtquantenhypothese, Zeitschrift f{\"u}r Physik
  26~(1) (1924) 178--181.

\bibitem{einstein2005quantentheorie}
A.~Einstein, Quantentheorie des einatomigen idealen gases. zweite abhandlung,
  Albert Einstein: Akademie-Vortr{\"a}ge: Sitzungsberichte der Preu{\ss}ischen
  Akademie der Wissenschaften 1914--1932 (2005) 245--257.

\bibitem{ott2001bose}
H.~Ott, J.~Fortagh, G.~Schlotterbeck, A.~Grossmann, C.~Zimmermann,
  Bose-einstein condensation in a surface microtrap, Physical review letters
  87~(23) (2001) 230401.

\bibitem{lieb2002proof}
E.~H. Lieb, R.~Seiringer, Proof of bose-einstein condensation for dilute
  trapped gases, Physical review letters 88~(17) (2002) 170409.

\bibitem{markowich2003}
W.~Bao, D.~Jaksch, P.~A. Markowich, Numerical solution of the gross--pitaevskii
  equation for bose--einstein condensation, Journal of Computational Physics
  187~(1) (2003) 318--342.

\bibitem{jia2022expansion}
F.~Jia, Z.~Huang, L.~Qiu, R.~Zhou, Y.~Yan, D.~Wang, Expansion dynamics of a
  shell-shaped bose-einstein condensate, Physical Review Letters 129~(24)
  (2022) 243402.

\bibitem{carollo2022observation}
R.~A. Carollo, D.~C. Aveline, B.~Rhyno, S.~Vishveshwara, C.~Lannert, J.~D.
  Murphree, E.~R. Elliott, J.~R. Williams, R.~J. Thompson, N.~Lundblad,
  Observation of ultracold atomic bubbles in orbital microgravity, Nature
  606~(7913) (2022) 281--286.

\bibitem{tononi2024quantum}
A.~Tononi, L.~Salasnich, A.~Yakimenko, Quantum vortices in curved geometries,
  AVS Quantum Science 6~(3) (2024).

\bibitem{van2020perturbation}
R.~A. Van~Gorder, Perturbation theory for bose--einstein condensates on bounded
  space domains, Proceedings of the Royal Society A 476~(2243) (2020) 20200674.

\bibitem{van2025geometrically}
R.~A. Van~Gorder, Geometrically confined three-dimensional bose--einstein
  condensates, Proceedings of the Royal Society A 481~(2318) (2025) 20240871.

\bibitem{erdHos2010derivation}
L.~Erd{\H{o}}s, B.~Schlein, H.-T. Yau, Derivation of the gross-pitaevskii
  equation for the dynamics of bose-einstein condensate, Annals of mathematics
  (2010) 291--370.

\bibitem{denschlag2002abose}
J.~H. Denschlag, J.~E. Simsarian, H.~H{\"a}ffner, C.~McKenzie, A.~Browaeys,
  D.~Cho, K.~Helmerson, S.~L. Rolston, W.~D. Phillips, Abose-einstein
  condensate in an optical lattice, Journal of Physics B: Atomic, Molecular and
  Optical Physics 35~(14) (2002) 3095.

\bibitem{cances2010numerical}
E.~Canc{\`e}s, R.~Chakir, Y.~Maday, Numerical analysis of nonlinear eigenvalue
  problems, Journal of Scientific Computing 45~(1) (2010) 90--117.

\bibitem{henning2020sobolev}
P.~Henning, D.~Peterseim, Sobolev gradient flow for the gross--pitaevskii
  eigenvalue problem: Global convergence and computational efficiency, SIAM
  Journal on Numerical Analysis 58~(3) (2020) 1744--1772.

\bibitem{bao2025computing}
W.~Bao, Z.~Chang, X.~Zhao, Computing ground states of bose-einstein
  condensation by normalized deep neural network, Journal of Computational
  Physics 520 (2025) 113486.

\bibitem{bao2003ground}
W.~Bao, W.~Tang, Ground-state solution of bose--einstein condensate by directly
  minimizing the energy functional, Journal of Computational Physics 187~(1)
  (2003) 230--254.

\bibitem{caliari2009minimisation}
M.~Caliari, A.~Ostermann, S.~Rainer, M.~Thalhammer, A minimisation approach for
  computing the ground state of gross--pitaevskii systems, Journal of
  Computational Physics 228~(2) (2009) 349--360.

\bibitem{wu2017regularized}
X.~Wu, Z.~Wen, W.~Bao, A regularized newton method for computing ground states
  of bose--einstein condensates, Journal of Scientific Computing 73~(1) (2017)
  303--329.

\bibitem{cances2014perturbation}
{\'E}.~Canc{\`e}s, G.~Dusson, Y.~Maday, B.~Stamm, M.~Vohral{\'\i}k, A
  perturbation-method-based a posteriori estimator for the planewave
  discretization of nonlinear schr{\"o}dinger equations, Comptes Rendus.
  Math{\'e}matique 352~(11) (2014) 941--946.

\bibitem{jarlebring2014inverse}
E.~Jarlebring, S.~Kvaal, W.~Michiels, An inverse iteration method for
  eigenvalue problems with eigenvector nonlinearities, SIAM Journal on
  Scientific Computing 36~(4) (2014) A1978--A2001.

\bibitem{adhikari2000numerical}
S.~K. Adhikari, Numerical solution of the two-dimensional gross--pitaevskii
  equation for trapped interacting atoms, Physics Letters A 265~(1-2) (2000)
  91--96.

\bibitem{edwards1995numerical}
M.~Edwards, K.~Burnett, Numerical solution of the nonlinear schr{\"o}dinger
  equation for small samples of trapped neutral atoms, Physical Review A 51~(2)
  (1995) 1382.

\bibitem{chang2007computing}
S.-L. Chang, C.-S. Chien, B.-W. Jeng, Computing wave functions of nonlinear
  schr{\"o}dinger equations: a time-independent approach, Journal of
  Computational Physics 226~(1) (2007) 104--130.

\bibitem{chiofalo2000ground}
M.~L. Chiofalo, S.~Succi, M.~Tosi, Ground state of trapped interacting
  bose-einstein condensates by an explicit imaginary-time algorithm, Physical
  Review E 62~(5) (2000) 7438.

\bibitem{bao2004computing}
W.~Bao, Q.~Du, Computing the ground state solution of bose--einstein
  condensates by a normalized gradient flow, SIAM Journal on Scientific
  Computing 25~(5) (2004) 1674--1697.

\bibitem{garcia2001optimizing}
J.~J. Garc{\'\i}a-Ripoll, V.~M. P{\'e}rez-Garc{\'\i}a, Optimizing
  schr{\"o}dinger functionals using sobolev gradients: applications to quantum
  mechanics and nonlinear optics, SIAM journal on scientific computing 23~(4)
  (2001) 1316--1334.

\bibitem{danaila2010new}
I.~Danaila, P.~Kazemi, A new sobolev gradient method for direct minimization of
  the gross--pitaevskii energy with rotation, SIAM Journal on Scientific
  Computing 32~(5) (2010) 2447--2467.

\bibitem{danaila2017computation}
I.~Danaila, B.~Protas, Computation of ground states of the gross--pitaevskii
  functional via riemannian optimization, SIAM Journal on Scientific Computing
  39~(6) (2017) B1102--B1129.

\bibitem{henning2025gross}
P.~Henning, E.~Jarlebring, The gross--pitaevskii equation and eigenvector
  nonlinearities: numerical methods and algorithms, SIAM Review 67~(2) (2025)
  256--317.

\bibitem{antoine2018asymptotic}
X.~Antoine, F.~Hou, E.~Lorin, Asymptotic estimates of the convergence of
  classical schwarz waveform relaxation domain decomposition methods for
  two-dimensional stationary quantum waves, ESAIM: Mathematical Modelling and
  Numerical Analysis 52~(4) (2018) 1569--1596.

\bibitem{yang2025energy}
L.~Yang, X.-G. Li, W.~Yan, R.~Zhang, The energy-diminishing weak galerkin
  finite element method for the computation of ground state and excited states
  in bose-einstein condensates, Journal of Computational Physics 520 (2025)
  113497.

\bibitem{ming2014efficient}
J.~Ming, Q.~Tang, Y.~Zhang, An efficient spectral method for computing dynamics
  of rotating two-component bose--einstein condensates via coordinate
  transformation, Journal of Computational Physics 258 (2014) 538--554.

\bibitem{aftalion2004giant}
A.~Aftalion, I.~Danaila, Giant vortices in combined harmonic and quartic traps,
  Physical Review A—Atomic, Molecular, and Optical Physics 69~(3) (2004)
  033608.

\bibitem{gammal1999improved}
A.~Gammal, T.~Frederico, L.~Tomio, Improved numerical approach for the
  time-independent gross-pitaevskii nonlinear schr{\"o}dinger equation,
  Physical Review E 60~(2) (1999) 2421.

\bibitem{bao2012mathematical}
W.~Bao, Y.~Cai, Mathematical theory and numerical methods for bose-einstein
  condensation, arXiv preprint arXiv:1212.5341 (2012).

\bibitem{osher1988fronts}
S.~Osher, J.~A. Sethian, Fronts propagating with curvature-dependent speed:
  Algorithms based on hamilton-jacobi formulations, Journal of computational
  physics 79~(1) (1988) 12--49.

\bibitem{gibou2018review}
F.~Gibou, R.~Fedkiw, S.~Osher, A review of level-set methods and some recent
  applications, Journal of Computational Physics 353 (2018) 82--109.

\bibitem{fedkiw1999non}
R.~P. Fedkiw, T.~Aslam, B.~Merriman, S.~Osher, A non-oscillatory eulerian
  approach to interfaces in multimaterial flows (the ghost fluid method),
  Journal of computational physics 152~(2) (1999) 457--492.

\bibitem{lee2023ghost}
H.~Lee, Y.~Liu, A ghost-point based second order accurate finite difference
  method on uniform orthogonal grids for electromagnetic scattering around
  curved perfect electric conductors with corners, Journal of Computational
  Physics 490 (2023) 112314.

\bibitem{bao2007energy}
W.~Bao, F.~Y. Lim, Y.~Zhang, Energy and chemical potential asymptotics for the
  ground state of bose-einstein condensates in the semiclassical regime,
  BULLETIN-INSTITUTE OF MATHEMATICS ACADEMIA SINICA 2~(2) (2007) 495.

\bibitem{bao2003numerical}
W.~Bao, S.~Jin, P.~A. Markowich, Numerical study of time-splitting spectral
  discretizations of nonlinear schr{\"o}dinger equations in the semiclassical
  regimes, SIAM Journal on Scientific Computing 25~(1) (2003) 27--64.

\bibitem{shishkin1990grid}
G.~Shishkin, Grid approximation of singularly perturbed elliptic equation in
  domain with characteristic bounds, Sov. J. Numer. Anal. Math. Modelling 5~(4)
  (1990) 327--343.

\bibitem{allgower1980simplicial}
E.~Allgower, K.~Georg, Simplicial and continuation methods for approximating
  fixed points and solutions to systems of equations, Siam review 22~(1) (1980)
  28--85.

\bibitem{zhao2005fast}
H.~Zhao, A fast sweeping method for eikonal equations, Mathematics of
  computation 74~(250) (2005) 603--627.

\bibitem{pan2018high}
S.~Pan, X.~Lyu, X.~Y. Hu, N.~A. Adams, High-order time-marching
  reinitialization for regional level-set functions, Journal of Computational
  Physics 354 (2018) 311--319.

\bibitem{macklin2006improved}
P.~Macklin, J.~Lowengrub, An improved geometry-aware curvature discretization
  for level set methods: application to tumor growth, journal of Computational
  Physics 215~(2) (2006) 392--401.

\bibitem{russell2003cartesian}
D.~Russell, Z.~J. Wang, A cartesian grid method for modeling multiple moving
  objects in 2d incompressible viscous flow, Journal of Computational Physics
  191~(1) (2003) 177--205.

\bibitem{gibou2013high}
F.~Gibou, C.~Min, R.~Fedkiw, High resolution sharp computational methods for
  elliptic and parabolic problems in complex geometries, Journal of Scientific
  Computing 54 (2013) 369--413.

\bibitem{gustafsson1978class}
I.~Gustafsson, A class of first order factorization methods, BIT Numerical
  Mathematics 18~(2) (1978) 142--156.

\bibitem{aslam2004partial}
T.~D. Aslam, A partial differential equation approach to multidimensional
  extrapolation, Journal of Computational Physics 193~(1) (2004) 349--355.

\bibitem{min2007geometric}
C.~Min, F.~Gibou, Geometric integration over irregular domains with application
  to level-set methods, Journal of Computational Physics 226~(2) (2007)
  1432--1443.

\bibitem{grundmann1978invariant}
A.~Grundmann, H.-M. M{\"o}ller, Invariant integration formulas for the
  n-simplex by combinatorial methods, SIAM Journal on Numerical Analysis 15~(2)
  (1978) 282--290.

\bibitem{muryshev2002dynamics}
A.~Muryshev, G.~Shlyapnikov, W.~Ertmer, K.~Sengstock, M.~Lewenstein, Dynamics
  of dark solitons in elongated bose-einstein condensates, Physical review
  letters 89~(11) (2002) 110401.

\bibitem{luckins2018bose}
E.~K. Luckins, R.~A. Van~Gorder, Bose--einstein condensation under the
  cubic--quintic gross--pitaevskii equation in radial domains, Annals of
  Physics 388 (2018) 206--234.

\bibitem{veksler2014simple}
H.~Veksler, S.~Fishman, W.~Ketterle, Simple model for interactions and
  corrections to the gross-pitaevskii equation, Physical Review A 90~(2) (2014)
  023620.

\bibitem{bao2019ground}
W.~Bao, Y.~Cai, X.~Ruan, Ground states of bose--einstein condensates with
  higher order interaction, Physica D: Nonlinear Phenomena 386 (2019) 38--48.

\bibitem{ruan2018normalized}
X.~Ruan, A normalized gradient flow method with attractive--repulsive splitting
  for computing ground states of bose--einstein condensates with higher-order
  interaction, Journal of Computational Physics 367 (2018) 374--390.

\bibitem{bao2019computing}
W.~Bao, X.~Ruan, Computing ground states of bose--einstein condensates with
  higher order interaction via a regularized density function formulation, SIAM
  Journal on Scientific Computing 41~(6) (2019) B1284--B1309.

\bibitem{heid2021gradient}
P.~Heid, B.~Stamm, T.~P. Wihler, Gradient flow finite element discretizations
  with energy-based adaptivity for the gross-pitaevskii equation, Journal of
  computational physics 436 (2021) 110165.

\bibitem{betcke2005reviving}
T.~Betcke, L.~N. Trefethen, Reviving the method of particular solutions, SIAM
  review 47~(3) (2005) 469--491.

\bibitem{engstrom2022higher}
C.~Engstr{\"o}m, S.~Giani, L.~Grubi{\v{s}}i{\'c}, Higher order composite dg
  approximations of gross--pitaevskii ground state: Benchmark results and
  experiments, Journal of computational and applied mathematics 400 (2022)
  113652.

\bibitem{zhang2016krylov}
R.~Zhang, J.~Zhu, X.-G. Li, A.~F. Loula, X.~Yu, A krylov semi-implicit
  discontinuous galerkin method for the computation of ground and excited
  states in bose--einstein condensates, Applied Mathematical Modelling 40~(7-8)
  (2016) 5096--5110.

\bibitem{salasnich2022bose}
L.~Salasnich, Bose-einstein condensate in an elliptical waveguide, SciPost
  Physics Core 5~(1) (2022) 015.

\bibitem{adhikari2012dipolar}
S.~K. Adhikari, Dipolar bose-einstein condensate in a ring or in a shell,
  Physical Review A—Atomic, Molecular, and Optical Physics 85~(5) (2012)
  053631.

\bibitem{golkebiewski2023spin}
M.~Go{\l}{\k{e}}biewski, H.~Reshetniak, U.~Makartsou, M.~Krawczyk, A.~Van
  Den~Berg, S.~Ladak, A.~Barman, Spin-wave spectral analysis in crescent-shaped
  ferromagnetic nanorods, Physical Review Applied 19~(6) (2023) 064045.

\bibitem{kwak2023minimum}
H.~Kwak, J.~H. Jung, Y.-i. Shin, Minimum critical velocity of a gaussian
  obstacle in a bose-einstein condensate, Physical Review A 107~(2) (2023)
  023310.

\end{thebibliography}

\end{document}